\documentclass[11pt, reqno]{amsart}
\usepackage{amsfonts}
\usepackage{mathrsfs}
\usepackage[active]{srcltx}
\usepackage{mathrsfs}
\usepackage{mathtools}
\usepackage{longtable}
 
\usepackage{amsmath,amsfonts,latexsym,graphicx,amssymb,url}
\usepackage{amsmath
}
\usepackage[active]{srcltx}
\usepackage{wasysym,pstricks, enumerate}
\usepackage{lscape,color,subfigure,graphicx,wrapfig}
\numberwithin{equation}{section}
\usepackage{mathtools}
\usepackage{enumerate}
\usepackage{enumitem}
\usepackage{mathrsfs}
\usepackage{bbm}
\usepackage{float}

\newcommand{\diff}{\,\mathrm{d}}
\DeclareMathOperator{\vol}{Vol}

\usepackage{xcolor}
\definecolor{bluecite}{HTML}{0875b7}
\usepackage[unicode=true,
bookmarksopen={true},
pdffitwindow=true,
colorlinks=true,
linkcolor=bluecite,
citecolor=bluecite,
urlcolor=bluecite,
hyperfootnotes=false,
pdfstartview={FitH},
pdfpagemode= UseNone]{hyperref}

\newtheorem{proposition}{Proposition}[section]
\newtheorem{theorem}{Theorem}[section]

\newtheorem{corollary}{Corollary}[section]
\theoremstyle{definition}
\newtheorem{definition}{Definition}[section]
\newtheorem{remark}{Remark}[section]
\numberwithin{equation}{section}

\usepackage{bbm}
\usepackage{dsfont}

\pagebreak
\address{\textsc{S\'andor Kaj\'ant\'o}:
	Department of Mathematics, Babe\c s-Bolyai University, Cluj-Napoca, Romania
}
\email{kajanto.sandor@math.ubbcluj.ro
}

\address{\textsc{Alexandru Krist\'aly}:
	Department of Economics, Babe\c s-Bolyai University, Cluj-Napoca, Romania  \& Institute of Applied Mathematics, \'Obuda
	University,
	Budapest, Hungary
}
\email{alex.kristaly@econ.ubbcluj.ro; kristaly.alexandru@uni-obuda.hu
}

\address{\textsc{Ioan Radu Peter}:
 Department  of Mathematics, Technical University of Cluj-Napoca, Memorandumului 28, RO-400114, Cluj-Napoca, Romania} \email
{ioan.radu.peter@math.utcluj.ro}

\address{\textsc{Wei Zhao}:
Department of Mathematics,
East China University of Science and Technology,
Shanghai, 200237,  China}
\email{szhao\underline{ }wei@yahoo.com}


\thanks{S. Kaj\'ant\'o  was supported by the UEFISCDI/CNCS grant PN-III-P4-ID-PCE2020-1001. A. Krist\'aly  was supported by the Excellence Researcher Program \'OE-KP-2-2022 of  \'Obuda University, Hungary, and by the 
	 UEFISCDI/CNCS grant PN-III-P4-ID-PCE2020-1001, Romania. W. Zhao was supported by Natural Science Foundation of Shanghai (No. 21ZR1418300).
}

\title[Riccati pairs]{A generic functional inequality and Riccati pairs: \\an alternative approach to Hardy-type inequalities}

\date{\today}

\author[S. Kaj\'ant\'o, A.\ Krist\'aly, I.R. Peter and W. Zhao]{S\'andor Kaj\'ant\'o, Alexandru Krist\'aly, Ioan Radu Peter and Wei Zhao}

\begin{document}
\begin{abstract}
	We present a generic functional inequality on Riemannian manifolds, both in \emph{additive} and \emph{multiplicative} forms, that produces well known  and genuinely new Hardy-type inequalities.
	For the additive  version, we introduce \emph{Riccati pairs} that extend Bessel pairs developed by Ghoussoub and Moradifam (\textit{Proc.\ Natl.\ Acad.\ Sci.\ USA}, 2008 \& \textit{Math.\ Ann.}, 2011). This concept enables us to give very short/elegant proofs of a number of celebrated functional inequalities  on Riemannian manifolds with sectional curvature bounded from above by simply solving a Riccati-type ODE. Among others,  we provide alternative proofs for Caccioppoli inequalities, Hardy-type inequalities and their improvements,  spectral gap estimates, interpolation inequalities, and Ghoussoub-Moradifam-type weighted inequalities. Concerning the multiplicative form, we prove sharp uncertainty principles on Cartan-Hadamard manifolds, i.e., Heisenberg-Pauli-Weyl uncertainty principles, Hydrogen uncertainty principles and Caffarelli-Kohn-Nirenberg inequalities. Some sharpness and rigidity phenomena are also discussed.

\end{abstract}
	\maketitle
	\vspace{-1cm}
	\tableofcontents
	\vspace{-1.4cm}
%\newpage
	\section{Introduction}

More than one hundred years elapsed since the first form of the celebrated \textit{Hardy inequality} appeared (sometimes called also as an \textit{uncertainty principle}), see Hardy \cite{Hardy}. It states that the $L^2$-norm of the singular term
$u(x)/|x|$ is controlled by the $L^2$-Dirichlet norm of the function $u\in C_0^\infty(\mathbb R^n)$. Several extensions and improvements of the Hardy inequality can be found by now in the literature, which became indispensable from the point of view of applications; indeed, solutions of a large class of elliptic problems involving singular terms are based on the validity of Hardy inequalities.\
Comprehensive discussions concerning these inequalities can be found in the monographs by Balinsky,  Evans and Lewis \cite{BalEvLew1}, Ghoussoub and Moradifam \cite{GM-book}, and  Ruzhansky and Suragan \cite{Ruzhansky-book}.

Generally speaking,  Hardy-type inequalities can be written in the following forms:
\begin{align}
	\label{eq:h}\int_\Omega V|\nabla u|^p\diff {\sf m}&\geq \int_\Omega W|u|^p\diff {\sf m},\quad \forall u\in C_0^\infty(\Omega),\tag{{\bf H}}\\
	\label{eq:up}\left(\int_\Omega V|\nabla u|^p\diff {\sf m}\right)^\frac{1}{p}\left(\int_\Omega W_1|u|^p\diff{\sf m}\right)^\frac{1}{p'}   &\geq \int_\Omega W_2|u|^p\diff {\sf m},\quad \forall u\in C_0^\infty(\Omega),\tag{{\bf UP}}
\end{align}
where $p>1$,  $p'=\frac{p}{p-1}$ is the \emph{conjugate} of $p$, $\Omega $ is an open subset of an ambient space $M$ (which could be the Euclidean space $\mathbb R^n$, any Riemannian/Finsler manifold, or a stratified group), ${\sf m}$ is a measure on $M$ and $V,W,W_1,W_2\colon\Omega\to (0,\infty)$ are certain potentials, possibly containing singular terms. 

Key observations related to problem \eqref{eq:h} have been made (mostly for $p=2$) by Adimurthi,  Chaudhuri and Ramaswamy  \cite{ACR},  Brezis and Marcus \cite{BM}, Brezis and V\'azquez \cite{BV}, Devyver, Fraas and Pinchover~\cite{DFP}, Fefferman \cite{Fefferman}, Filippas,  Maz'ya and Tertikas \cite{FilMazTert1, FilMazTert3}, Filippas and Tertikas \cite{Fil-Ter},
 Muckenhoupt \cite{Muckenhoupt},  Ruzhansky and Suragan \cite{Ruzhansky-AdvMath}, Tertikas and Zographopoulos \cite{Tertikas}, etc.
% in fact, all these results provided sufficient criteria to guarantee \eqref{eq:h} for  special choices of $V$ and $W$.

A milestone  result -- concerning problem  \eqref{eq:h} -- has been provided by Ghoussoub and  Moradifam \cite{GM-Annalen, GM-PNAS} (for $p=2$ and the Euclidean setting), where the authors proved that \eqref{eq:h} holds if and only if  $(V,W)$ is a \textit{Bessel pair}. The latter notion is based on the solvability of a second order linear ODE containing the potentials $V$ and $W$. We note that this ODE agrees with the radial equation obtained from the classical approach of \emph{supersolutions}, which emerges from the early works of Allegretto~\cite{All} and Moss and Piepenbrink~\cite{Moss}, see Remark~\ref{rem:supersolution}.  The concept of Bessel pairs was extended to general $p>1$, see Duy, Lam and Lu \cite{DuyLL}, and also has applications on nonpositively curved Riemannian manifolds, see Flynn,  Lam,  Lu and Mazumdar \cite{Flynn-JGA}, where still the usual notion of Bessel pairs and fine comparison arguments are used. 

In this paper we use a genuinely different approach to prove Hardy inequalities. First we provide a generic inequality on Riemannian manifolds in both \emph{additive} and \emph{multiplicative} forms, that turn out to be \emph{equivalent} to each other, producing inequalities of type \eqref{eq:h} and \eqref{eq:up}, respectively. Since these inequalities contain the Laplacian of a given potential, -- which implicitly encode curvature information about the manifold -- an appropriate comparison argument furnishes in the additive form a  Riccati-type ordinary differential inequality  that leads us to the notion of \textit{Riccati pair} for certain potentials. This notion turns out to be extremely efficient to prove inequalities of the type  \eqref{eq:h}.
It is worth to be pointed out that Bobkov and Götze~\cite{Bobkov-Gotze} already provided some integral estimates on certain domains  by using 
canonical Riccati and Sturm-Liouville equations on the real line.
% which is based on the solvability of a Riccati-type ODE. This notion turned out to be extremely efficient to prove inequalities of the type  \eqref{eq:h}. In fact, we shall prove throughout various examples that classical results from the literature  -- and not only -- follow by short/elegant arguments. 

To present our approach, we put ourselves into the realm of Riemannian manifolds. Let $\nabla_g$, $\Delta_g$,  $|\cdot|$ and ${\rm d}v_g$ be the gradient, Laplace-Beltrami operator, norm and standard volume form on a Riemannian manifold $(M,g)$, respectively. Here is our first main result containing the generic inequality (for a weighted version, see Theorem \ref{theorem-weighted-main}): 
	
\begin{theorem}\label{theorem-main-0}
	Let $(M,g)$ be a complete, non-compact $n$-dimensional Riemannian manifold, with  $n\geq 2$. Let  $\Omega\subseteq M$ be a domain, $p>1$, and $\rho\in W_{\rm loc}^{1,p}(\Omega)$ be a positive function with $|\nabla_g\rho|=1$  $\diff v_g$-a.e.\ in $\Omega$. Let  $G\colon(0,\sup_\Omega\rho)\to \mathbb R$ and $H\colon \mathbb{R}\to\mathbb{R}$  be  $C^1$  functions such that
	\begin{enumerate}[label=\rm({\bf G})$_\rho$]
		\item\label{item:C_0-no-w-0} {\rm :} $G(\rho)\in L_{\rm loc}^{p'}(\Omega)$, $G'(\rho)\in L_{\rm loc}^{1}(\Omega),$
	\end{enumerate}
	and $H(0)=H'(0)=0$. The following inequalities hold.
	\begin{itemize}
		\item[{\rm (i)}] {\rm (Additive form)} For every  $u\in C_{0}^\infty(\Omega)$ one has
		\begin{equation}
			\int_\Omega	|\nabla_gu|^{p} \diff v_g \geq p \int_\Omega  \left[G'(\rho)+ G(\rho)\Delta_{g}\rho\right]H(u)\diff v_g {-(p-1)}\int_\Omega |G(\rho)|^{p'} |H'(u)|^{p'}\diff v_g.\label{additive-no-weight-0}
		\end{equation}
		\item[{\rm (ii)}] {\rm (Multiplicative form)} For every  $u\in C_{0}^\infty(\Omega)$ one has
		\begin{equation}\label{multiplicative-no-weight-0}
			\int_\Omega	|\nabla_gu|^{p}  \diff v_g \geq  \frac{\left|\displaystyle\int_\Omega  \left[G'(\rho)+ G(\rho)\Delta_{g}\rho\right]H(u)\diff v_g\right|^p}{\left(\displaystyle\int_\Omega |G(\rho)|^{p'} |H'(u)|^{p'}\diff v_g\right)^{p-1}},
		\end{equation}
		provided that there exists a neighborhood $U\subset \mathbb R$ of zero, such that $H'(s)\neq 0$, for every $s\in U\setminus \{0\}$ and $G(t)\neq 0$ for all $t\in (0,\sup_\Omega\rho)$. 
		%{\color{teal}esetleg keresztbe szorozva?}
	\end{itemize}
\end{theorem}

A seemingly surprising fact emerges: it turns out that  \eqref{additive-no-weight-0} and \eqref{multiplicative-no-weight-0} are \emph{equivalent} in the sense that they can be deduced from each other, see Remark \ref{rem:equivalence}. Accordingly, we may roughly say that  these two classes of functional inequalities have the same `algebraic' origin. This fact follows from the `duality' between $H(u)$ and $|H'(u)|^{p'}$, which can be viewed as a competition between the two terms, implying  the aforementioned equivalence through a simple scaling argument.

%It is worth to point out the `duality' between $H(u)$ and $|H'(u)|^{p'}$, which can be viewed also as a competition between the two terms. In particular, this competition phenomena combined with a simple scaling argument show  the hidden fact that \eqref{additive-no-weight-0} and \eqref{multiplicative-no-weight-0} are \emph{equivalent} in the sense that they can be deduced from each other, see Remark \ref{rem:equivalence}. Accordingly, we may roughly say that  these two classes of functional inequalities have the same `algebraic' origin.

On the one hand, both inequalities in Theorem \ref{theorem-main-0} are \emph{generic}, that is for suitable choices of functions $\rho$, $G$ and $H$ they produce various functional inequalities. While the additive form  \eqref{additive-no-weight-0} is tailored to produce \textit{Hardy-type inequalities} in the form of \eqref{eq:h},  see \S \ref{section-Hardy}, the multiplicative form \eqref{multiplicative-no-weight-0} will provide various \textit{uncertainty principles} of the type \eqref{eq:up}, see \S \ref{section-uncertainty}. 
%, to achieve this, one needs to multiply both sides with de denominator, and raise them to the power $1/p$. {\color{teal} Kihagyható ez a mondat}

On the other hand, the aforementioned procedure can be reversed: for a number of Hardy-type inequalities one can find suitable choices of functions providing short/elegant proofs for them. To see this, let us focus on the additive form \eqref{additive-no-weight-0}. Let $p>1$ and observe that if $H(s)=|s|^p/p$ for all $s\in \mathbb R$, then $pH(s)=|H'(s)|^{p'}=|s|^p.$ This observation and the Laplace comparison (see Theorem \ref{comparison-theorem}) suggest the following notion (for a weighted form see Definition \ref{def:wRP}); hereafter,  $\mathcal H^n$ stands for the $n$-dimensional Hausdorff measure.

\begin{definition} Let $(M,g)$ be a complete, non-compact $n$-dimensional Riemannian manifold, with  $n\geq 2$. Let  $\Omega\subseteq M$ be a domain, $p>1$, and $\rho\in W_{\rm loc}^{1,p}(\Omega)$ be a positive function with $|\nabla_g\rho|=1$  $\diff v_g$-a.e.\ in $\Omega$.  Let us fix the continuous functions   $L,W\colon(0,\sup_\Omega\rho)\to (0,\infty)$.\
	We say that the couple $(L,W)$ is a $(p,\rho)$-\textit{Riccati pair  in} $(0,\sup_\Omega\rho)$ if there exists a  function $G\colon(0,\sup_\Omega\rho)\to \mathbb R$  such that 	
	\begin{enumerate}[label=\bf (r\arabic*)]
		\item\label{it:r1} \ref{item:C_0-no-w-0} holds (from Theorem \ref{theorem-main-0});
		\item\label{it:r2} $\Delta_g \rho\geq L(\rho)$ in the distributional sense in $\Omega,$ and  $G\geq 0$ if $\mathcal H^n(\{x\in \Omega:\Delta_g \rho(x) > L(\rho(x))\})\neq 0;$
		\item\label{it:r3} for every $t\in (0,\sup_\Omega\rho)$ one has
	\begin{equation}\label{Riccati pair-0-0}
		G'(t)+L(t)G(t){-(p-1)}|G(t)|^{p'}\geq  W(t).
	\end{equation}
	\end{enumerate}
	If a function $G$ satisfies the above conditions, then it is said to be \emph{admissible} for $(L,W)$.
\end{definition}

An efficient application of Theorem \ref{theorem-main-0} can be stated as follows (for a weighted form, see Theorem \ref{W-theorem}):

%holds on \textit{Cartan-Hadamard manifolds} (i.e., complete, simply connected Riemannian manifolds with nonpositive sectional curvature). Accordingly, let $(M,g)$ be a Cartan-Hadamard manifold with the induced metric $d_g$. In the sequel, for $x_0\in M$, we shall denote
%$d_{x_0}(x)=d_g(x_0,x)$, $x\in M$. 	
	
%	we say that the couple
%		 a pair of functions  with   nonnegative $W:{\rm Im}\rho\to \mathbb R$ and a nonpositive differentiable function  such that the superposition property $({\rm HG})$ hold with $H(s)=\frac{|s|^p}{p},$ $s\in \mathbb R$, and

\begin{theorem}\label{Riccati-theorem-0}
	Let $(M,g)$ be a complete, non-compact $n$-dimensional Riemannian manifold, with  $n\geq 2$. Let  $\Omega\subseteq M$ be a domain, $p>1$, and $\rho\in W_{\rm loc}^{1,p}(\Omega)$ be a positive function with $|\nabla_g\rho|=1$  $\diff v_g$-a.e.\ in $\Omega$. Let $L,W\colon(0,\sup_\Omega\rho)\to (0,\infty)$ be continuous functions such that $(L,W)$ is a $(p,\rho)$-Riccati pair in $(0,\sup_\Omega\rho)$.
	Then for every
 $u\in C_0^{\infty}(\Omega)$ one has \begin{equation}\label{Riccati-Reduced-0}
	\int_\Omega	|\nabla_gu|^{p} \diff v_g \geq  \int_\Omega W(\rho)|u|^p    \diff v_g.
\end{equation}
%Let $(M,g)$ be a $n$-dimensional Riemannian manifold $(n\geq 2)$, let $\Omega\subseteq M$ be a domain,  and $x_0\in \Omega$ be fixed. Let $p>1$, $S=\{d_{x_0}(x):x\in \Omega\}$, and assume that
%\begin{itemize}
%	\item[{\rm (i)}] $\Delta_g \rho\geq L(\rho)$ in the distributional sense in $\Omega;$
%%	\item $G:S\to \mathbb R$ is a nonpositive differentiable function satisfying $({\rm G}),$
%%	\item $W:S\to \mathbb R$ is a nonnegative continuous function, and
%	\item[{\rm (ii)}]  $(L,W)$ is a $(p,\rho)$-Riccati pair in $(0,\sup_\Omega\rho)$.
%\end{itemize}
\end{theorem}

Just to give a little peek to the efficiency of Theorem \ref{Riccati-theorem-0}, we sketch a short proof of the celebrated \textit{McKean's sharp spectral gap estimate} (for  details, see  Theorem \ref{McKean-theorem}):  if the sectional curvature is bounded from above by $\kappa< 0$ on an $n$-dimensional Cartan-Hadamard manifold\footnote{$(M,g)$ is a Cartan-Hadamard manifold if it is a complete, simply connected Riemannian manifold with nonpositive sectional curvature.} $(M,g)$, the essential spectrum of the $p$-Laplace-Beltrami operator on $(M,g)$ is $[K_{\kappa,n,p},\infty)$, where  
$K_{\kappa,n,p}=(\frac{n-1}{p}\sqrt{-\kappa})^p.$ 
Indeed, choose $\rho=d_g(x_0,\cdot)$ for some $x_0\in M$, where $d_g$ is the usual distance function on $(M,g)$, and let 
$$L\equiv(n-1)\sqrt{-\kappa},\quad W\equiv(n-1)\sqrt{-\kappa}c{-(p-1)}c^{p'} \quad\mbox{and}\quad G\equiv c,$$ for some $c>0$. The Laplace comparison, see Theorem \ref{comparison-theorem}/(I)/(i), implies $\Delta_g \rho\ge L$, hence \ref{it:r2} holds. We obtain that  $(L,W)$ is a $(p,\rho)$-Riccati pair in $(0,\infty)$, and $G$ is admissible for $(L,W)$. Since the maximum of $(n-1)\sqrt{-\kappa}c{-(p-1)}c^{p'}$ in $c>0$ is precisely $K_{\kappa,n,p}$, inequality  \eqref{Riccati-Reduced-0} immediately implies the required spectral estimate.

When the domain $\Omega$ is embedded into any of the model space forms (i.e.,  the sphere,  the Euclidean space or the hyperbolic space) and $\rho$ is a distance function from a fixed point of the domain, a closer inspection of the proof of  Theorem \ref{Riccati-theorem-0} (and its  weighted version, Theorem \ref{W-theorem})  shows that the \textit{sign restriction}  on the function $G$ from \eqref{Riccati pair-0-0} can be \textit{relaxed}. Indeed, in this particular situation, it turns out that there is equality in the Laplace comparison (i.e.  $\Delta_g \rho= L(\rho)$ for a suitable choice of $L$), thus certain estimates will be independent on the sign of $G$; for more details see Remark \ref{remark-sign}.

A natural question arises concerning the connection between Riccati and Bessel pairs. In the Euclidean case, the weighted Riccati pair (see Definition \ref{def:wRP}) is a slight extension  of the Bessel pair; indeed,  if we choose $L(t)=(n-1)/t$, require that $G\ge0$ and restrict \eqref{Riccati pair-0-0} to the equality case, the two notions coincide, see Proposition \ref{Bessel-Riccati}. When we move to the Riemannian context, there exists a natural extension of the usual notion of Bessel pairs (using appropriate comparison $L$) which still preserves this equivalence whenever the sectional curvature is nonpositive, compare Remark \ref{bessel-pair-generalization} and the paper by Flynn,  Lam,  Lu and Mazumdar \cite{Flynn-JGA}. However, in the generic case, Riccati pairs `inherit' the geometry of the ambient Riemannian manifold, reflected in the curvature; see e.g.\ the alternative proof of Cheng's eigenvalue result (Theorem \ref{Cheng-theorem}). In fact, the use of Riccati pairs is twofold. First, Riccati pairs  \emph{naturally appear} from our generic inequality \eqref{additive-no-weight-0}, thus no redundant change of variables is required. Second, Riccati pairs are \emph{flexible}  from technical point of view as both $L$ and $W$ can be easily adjusted to meet our needs. Such an example is provided by the  proof of McKean's sharp spectral gap estimate, where the choice of $L\equiv(n-1)\sqrt{-\kappa}$ significantly reduces the computations comparing to the expected form $L(t)=(n-1)\sqrt{-\kappa}\coth(\sqrt{-\kappa} t)$.
Beyond Proposition~\ref{Bessel-Riccati}, we point out the relationship between the method of supersolutions, also known as the \emph{Allegretto-Moss-Piepenbrink approach}, and the aforementioned notions, see Remark~\ref{rem:supersolution}. These relations allow us to incorporate in our method various elements of the \emph{criticality theory} of Hardy inequalities developed in the approach of supersolutions, see the foundational works by Pinchover~\cite{Pinchover1,Pinchover2}, and Devyver, Fraas and Pinchover~\cite{DFP}. We refer to Remark~\ref{rem:criticality}/(b) for an illustration.

In the sequel we roughly describe the content of the paper. In section \ref{section-preliminaries} those notions and results are recalled that will be used in our arguments, as basic elements from Riemannian geometry as well as volume and Laplace comparison principles. Section \ref{section-main-result} is devoted to the proof of our main abstract results: Theorems \ref{theorem-weighted-main} \& \ref{W-theorem}, that are the weighted versions of Theorems \ref{theorem-main-0} \& \ref{Riccati-theorem-0}, respectively. In addition, the relationship between Bessel and Riccati pairs is also discussed, see Proposition \ref{Bessel-Riccati} and Remark \ref{bessel-pair-generalization}.

Most of the additive functional inequalities in section \ref{section-Hardy} are formally well known from the Euclidean setting (and some of them also in Riemannian manifolds). However, we shall focus on their proof on
Riemannian manifolds  -- mainly on Cartan-Hadamard manifolds -- by showing the efficiency of our main results, mostly based on Riccati pairs.  
More precisely, we shall consider the following inequalities:
\begin{itemize}
	\item In \S \ref{Caccippoli-subsection} we prove $L^p$-\textit{Caccioppoli-type inequalities} on Riemannian manifolds, providing alternative, short proofs for the results obtained by D'Ambrosio and Dipierro \cite{DDipierro} (see Theorems \ref{theorem-Hardy-1} \& \ref{theorem-Cacc-improv}). Some new improvements are also established in the case  $p\in (1,2].$
	\item In \S \ref{Hardy-subsection} various \textit{Hardy-type  inequalities and their improvements} are discussed on \textit{Cartan-Hadamard manifolds}, including results by Carron \cite{Carron} and  Kombe and \"Ozaydin \cite{KO-2009, KO-2013} (see Theorem \ref{them-hardy-1}),  Edmunds and Triebel \cite{EdmundsTriebel} (see Theorem \ref{them-elso}), Adimurthi,  Chaudhuri and  Ramaswamy \cite{ACR} (see Theorem \ref{adimurth-theorem}), and
	Brezis and V\'azquez \cite{BV} (see Theorem \ref{Brezis-Vazquez-theorem}).
	\item  In \S \ref{Spectral-estimate-subsection} \textit{spectral  estimates on Riemannian manifolds} are presented. First, an alternative proof of the celebrated Cheng's comparison result is provided on Riemannian manifolds with sectional curvature bounded from above (see Theorem \ref{Cheng-theorem}), then the well known Faber-Krahn inequality and McKean spectral gap estimates (see Theorems \ref{FK-theorem} \& \ref{McKean-theorem}) are considered. In addition, we give a short proof of the main spectral result of Carvalho and  Cavalcante \cite{Car-Cav-JMAA} (see Theorem \ref{Car-Cav-theorem}).
	\item In \S \ref{interpolation-mckean-hardy} an \textit{interpolation inequality} is established connecting the Hardy inequality and McKean's spectral gap estimate on \textit{Cartan-Hadamard manifolds}, in the spirit of Berchio,  Ganguly, Grillo and Pinchover \cite{Berchio-Royal} (see Theorem \ref{interpolation-theorem}). A simple modification of the latter argument also provides a short, alternative proof of the inequality by Akutagawa and Kumura \cite{Japanok}  (see Theorem \ref{japanok-cikke-alapjan}).
	\item In \S \ref{gh-morad-subsection} two parameter-dependent  \textit{Ghoussoub-Moradifam-type weighted inequalities} are considered  in the Euclidean case (cf.\ \cite{GM-Annalen}), where the weights are of non-singular type (see Theorem \ref{Gh-Moradifam-theorem}). For a certain parameter range, these inequalities are also extended to Cartan-Hadamard manifolds (see Theorem \ref{0-Gh-Moradifam-theorem-0}).
\end{itemize}
In section \ref{section-uncertainty} we prove multiplicative Hardy-type inequalities, as simple consequences of Theorem \ref{theorem-main-0}/(ii) and its weighted version, Theorem \ref{theorem-weighted-main}/(ii); namely:
\begin{itemize}
	\item In \S \ref{subsection-uncertainty-1} a  sharp parameter-dependent \textit{uncertainty principle} is established on \textit{Cartan-Hadamard manifolds}, which implies both the Heisenberg-Pauli-Weyl and Hydrogen uncertainty principles  (see Theorem \ref{Uncertainty-theorem}). In addition,  we prove a rigidity result: if the quantitative uncertainty principle holds on an $n$-dimensional Cartan-Hadamard manifolds with Ricci curvature bounded from below then the manifold is isometric the corresponding model space form (see Theorem \ref{Uncertainty-theorem-rigid}).
	\item In \S \ref{CKN-subsection} two sharp \textit{Caffarelli-Kohn-Nirenberg inequalities} are presented on \textit{Cartan-Hadamard manifolds} (see Theorems \ref{CKN-th} \& \ref{thm-final}).
\end{itemize}

It is well known that Bessel pairs can be efficiently used to prove various \emph{Hardy-Rellich-type inequalities}, see Ghoussoub and Moradifam \cite{GM-Annalen} and Berchio, Ganguly, Roychowdhury \cite{Berchio-CV}. We are wondering if such inequalities of  higher-order can be elegantly achieved by Riccati pairs, both in Euclidean spaces and Riemannian manifolds. This question will be addressed in a forthcoming paper.

We also notice that our arguments can be easily extended to reversible Finsler manifolds; in the irreversible case, there are certain technical issues which require the presence of the reversibility constant.  The details are left to the interested reader. 

	\section{Preliminaries}\label{section-preliminaries}
	
In this section we recall those notions and results that are needed for presenting our results. We mainly follow Gallot,  Hulin and  Lafontaine \cite{GHL}, Hebey \cite{Hebey}, and Krist\'aly \cite{Kristaly-JMPA}.

Let $(M,g)$ be an $n$-dimensional complete  Riemannian manifold, with $n\geq 2$,  $d_g\colon M\times M \to [0,\infty)$ and $\diff v_g$ denote the distance function and the canonical measure induced by the metric $g$ on $M$, respectively.
Denote $B_{x_0}(R)=\{y\in M:d_g(x_0,y)<R\}$ the open metric ball with center $x_0\in M$ and radius $R>0.$
The volume of a bounded open set $S\subset M$ is
$$\vol_g(S)=\int_S {\text d}v_g=\mathcal H^n(S).$$
%where ${\rm
%	Haus}_{d_g}(S)$ is the Hausdorff measure of $S$ with respect to the
%metric function $d_g$.
In particular, one has for every $x_0\in M$ that
\begin{equation}\label{volume-comp-nullaban}
	\lim_{R\to 0^+}\frac{\vol_g(B_{x_0}(R))}{\omega_n
		R^n}=1,
\end{equation}
where $\omega_n=\pi^\frac{n}{2}/\Gamma(1+n/2)$ stands for the volume of the $n$-dimensional
Euclidean unit ball, see Gallot,  Hulin and  Lafontaine \cite[Theorem 3.98]{GHL}.

Let $p>1$ and fix $u\in W^{1,p}(M)$. The gradient of $u$ is $\nabla_g u$, whose local components are $$u^i=g^{ij}\frac{\partial u}{\partial x^j},$$ where $g^{ij}$ are
the local components of $g^{-1}=(g_{ij})^{-1}$ in the local
coordinate system  $(x^i)$ on a coordinate neighborhood of $x\in M$; hereafter the standard summation convention is used.
The $p$-Laplace-Beltrami operator on $(M,g)$ is given by $$\Delta_{g,p} u={\rm
	div}_g(|\nabla_gu|^{p-2}\nabla_gu),$$ see e.g.\ Hebey \cite{Hebey}.  Observe that $\Delta_{g,2}$ is nothing but $\Delta_{g}$, the usual Laplace-Beltrami operator  on $(M,g)$. If $w\in C_0^2(M)$, one has the following integration by parts formula
\begin{equation}
	\int_M w \Delta_{g,p} u {\text d}v_g=-\int_M  |\nabla_gu|^{p-2} \nabla_g u \nabla_g w {\text d}v_g,
\end{equation}	
where $\nabla_gu,\nabla_gw\in TM,$ and $\nabla_g u \nabla_g w$ is understood as the scalar product with respect to the metric $g$.

% whose expression in a local chart of associated
%coordinates $(x^i)$ is $$\Delta_g u=g^{ij}\left(\frac{\partial^2
%	u}{\partial x_i\partial x_j}-\Gamma_{ij}^k\frac{\partial u}{\partial
%	x_k}\right),$$ where $\Gamma_{ij}^k$ are the coefficients of the
%Levi-Civita connection.

The model space form ${\bf M}_\kappa^n$ is an $n$-dimensional manifold with constant sectional curvature~$\kappa$, that is 
\[
	{\bf M}_\kappa^n = \begin{cases}
		\mathbb{S}_\kappa^n\mbox{ -- the $\kappa$-sphere}, & \mbox{if } \kappa>0,\\
		\mathbb{R}^n\mbox{ -- the Euclidean space}, & \mbox{if }\kappa = 0,\\
		\mathbb H^n_{\kappa}\mbox { -- the $\kappa$-hyperbolic space}, & \mbox{if } \kappa<0.
	\end{cases}
\]
For $\kappa\in \mathbb R$, we introduce the functions
$${\bf ct}_\kappa(t)=\begin{cases}
	\sqrt{\kappa}\cot(\sqrt{\kappa} t), &\mbox{if }\kappa>0,\\
	\frac{1}{t},& \mbox{if } \kappa=0,  \\
	\sqrt{-\kappa}\coth(\sqrt{-\kappa} t), & \mbox{if } \kappa<0,\\
\end{cases}\quad\mbox{and}\quad{\bf s}_\kappa(t)=\begin{cases}
	\frac{\sin(\sqrt{\kappa} t)}{\sqrt{\kappa}}, &\mbox{if }\kappa>0,\\
	t,& \mbox{if } \kappa=0,  \\
	\frac{\sinh(\sqrt{-\kappa} t)}{\sqrt{-\kappa}}, & \mbox{if } \kappa<0,
\end{cases}$$
defined on $(0,\frac{\pi}{\sqrt{\kappa}})$ using the convention $\frac{\pi}{\sqrt{\kappa}}:=\infty$, when $\kappa\le 0$. For $\kappa\leq 0$ define
${\bf D}_\kappa\colon[0,\infty)\to [0,\infty)$  by
\begin{equation}\label{D-function}
	{\bf D}_\kappa(0)=0\quad\mbox{and}\quad{\bf D}_\kappa(t)=t {\bf ct}_\kappa(t)-1,\quad\forall t>0.
\end{equation}
For further use, let
$$V_\kappa(R)=n\omega_n\int_0^R {\bf s}^{n-1}_\kappa(t)\diff t$$
be the volume of the ball of radius $R>0$ in the model space form ${\bf M}_\kappa^n$.

	Let $d_{x_0}(x)=d_g(x_0,x)$ be  the distance  from 	a given point $x_0\in M$. The eikonal equation reads as
	\begin{equation}\label{eikonal}
		|\nabla_g d_{x_0}|=1\ \ \diff v_g-{\rm a.e.\ on}\ M.
	\end{equation}

Let $\operatorname{inj}_{x_0}$ be the injectivity radius of $x_0\in M$.  We have the following Laplace and Bishop-Gromov volume comparison principles, see e.g.\ Gallot,  Hulin and  Lafontaine \cite[Theorem 3.101]{GHL}.

\begin{theorem}\label{comparison-theorem} Let $(M,g)$ be an $n$-dimensional complete  Riemannian manifold with $n\geq 2$, and  $x_0\in M$  be a fixed point. The following statements hold with the usual convention $(\pi/\sqrt{\kappa}=\infty$, when $\kappa\le 0)$.
	\begin{itemize}
		\item[{\rm (I)}] If the sectional curvature is bounded from above as ${\bf K}\leq  \kappa $ for some $\kappa\in \mathbb R$, then$:$
		\begin{itemize}
			\item[{\rm (i)}] {\rm (Laplace comparison)} $\Delta_g d_{x_0}\geq (n-1){\bf ct}_\kappa(d_{x_0})$ with $d_{x_0}<\pi/\sqrt{\kappa};$
			\item[{\rm (ii)}] {\rm (Volume comparison)} $\vol_g(B_{x_0}(R))\geq V_\kappa(R)$ for every $R\in(0,\min(\operatorname{inj}_{x_0},\pi/\sqrt{\kappa}))$. 
		\end{itemize}

		\item[{\rm (II)}] If the Ricci curvature is bounded from below as ${\bf Ric}\geq   \kappa (n-1)g $ for some $\kappa\in \mathbb R$, then$:$
		\begin{itemize}
			\item[{\rm (i)}] {\rm (Laplace comparison)} $\Delta_g d_{x_0}\leq (n-1){\bf ct}_\kappa(d_{x_0})$ with $d_{x_0}<\pi/\sqrt{\kappa};$
			\item[{\rm (ii)}] {\rm (Volume comparison)} $\vol_g(B_{x_0}(R))\leq V_\kappa(R)$ for every $R>0.$ 
		\end{itemize}
	\end{itemize}
	If equality holds in any of the above statements,  $(M,g)$ is isometric to the model space form ${\bf M}_\kappa^n$.
\end{theorem}

%	
%	\begin{theorem}\label{comparison-volume}{\rm [Comparison principles]}
%		Let $(M,g)$ be a complete, $n-$di\-men\-sional Riemannian manifold.
%		Then the following statements hold.
%		\begin{itemize}
%			\item[{\rm (a)}] If $(M,g)$ is a Cartan-Hadamard manifolds,
%			the function
%			$\rho\mapsto \frac{\vol_g(B(x,\rho))}{\rho^n}$ is
%			non-decreasing, $\rho>0$. % for $0<\rho\leq i_x$, where $i_x$ is the injectivity radius of $x\in M.$
%			In particular, from {\rm (\ref{volume-comp-nullaban})} we have
%			\begin{equation}\label{volume-comp-altalanos-0}
%				{\vol_g(B(x,\rho))}\geq \omega_n \rho^n\ {for\ all}\ x\in M\
%				{and}\ \rho>0.
%			\end{equation}
%			If equality holds in {\rm (\ref{volume-comp-altalanos-0})}, then the
%			sectional curvature is identically zero.
%			\item[{\rm (b)}] If $(M,g)$ has non-negative Ricci curvature,
%			the function
%			$\rho\mapsto \frac{\vol_g(B(x,\rho))}{\rho^n}$ is
%			non-increasing, $\rho>0$. In particular, from {\rm
%				(\ref{volume-comp-nullaban})} we have
%			\begin{equation}\label{volume-comp-altalanos-2}
%				{\vol_g(B(x,\rho))}\leq \omega_n \rho^n\ {for\ all}\ x\in M\
%				{and}\ \rho>0.
%			\end{equation}
%			If equality holds in {\rm (\ref{volume-comp-altalanos-2})}, then the
%			sectional curvature is identically zero.
%		\end{itemize}
%	\end{theorem}

	\section{Proof of main results}\label{section-main-result}
	
% Our main result reads as follows.
%\textbf{They are equivalent. }

In this section we prove the main abstract results in weighted form.
\begin{theorem}\label{theorem-weighted-main}
	Let $(M,g)$ be a complete, non-compact $n$-dimensional Riemannian manifold, with  $n\geq 2$. Let  $\Omega\subseteq M$ be a domain, $p>1$, and $\rho\in W_{\rm loc}^{1,p}(\Omega)$ be nonconstant and positive with $\mathcal H^n(\rho^{-1}(\sup_\Omega \rho))=0$.  
	Let  $w\colon(0,\sup_\Omega\rho)\to (0,\infty)$, $G\colon(0,\sup_\Omega\rho)\to \mathbb R$ and $H\colon \mathbb{R}\to\mathbb{R}$  be  $C^1$  functions such that
	\begin{enumerate}[label=\rm({\bf G})$_{\rho,w}$]
		\item\label{item:C_0} {\rm :} $G(\rho)  w(\rho)^\frac{1}{p'} |\nabla_g\rho|^{p-1}\in L_{\rm loc}^{p'}(\Omega)$ and $G'(\rho) w(\rho) |\nabla_g\rho|^{p},
		G(\rho) w'(\rho)|\nabla_g\rho|^{p}, w(\rho)\in L_{\rm loc}^{1}(\Omega),$
	\end{enumerate}
	and $H(0)=H'(0)=0$. The following inequalities hold.
	\begin{enumerate}
		\item[{\rm (i)}] {\rm (Additive form)} For every  $u\in C_{0}^\infty(\Omega)$ one has
		\begin{align}
			\int_\Omega	w(\rho) |\nabla_gu|^{p} \diff v_g &\geq p \int_\Omega  \left[ (G'(\rho) w(\rho)+ G(\rho) w'(\rho))  |\nabla_g\rho|^{p} +  G(\rho) w(\rho)\Delta_{g,p}\rho\right]H(u)\diff v_g \nonumber\\&\qquad{-(p-1)}\int_\Omega |G(\rho)|^{p'}|\nabla_g\rho|^{p} w(\rho) |H'(u)|^{p'}\diff v_g.\label{additive}
		\end{align}
		\item[{\rm (ii)}] {\rm (Multiplicative form)} For every  $u\in C_{0}^\infty(\Omega)$ one has
		\begin{equation}\label{multiplicative}
			\int_\Omega	w(\rho)|\nabla_gu|^{p}  \diff v_g \geq  \frac{\left|\displaystyle\int_\Omega  \left[ (G'(\rho) w(\rho)+ G(\rho) w'(\rho))  |\nabla_g\rho|^{p} +  G(\rho) w(\rho) \Delta_{g,p}\rho\right]H(u)\diff v_g\right|^p}{\left(\displaystyle\int_\Omega |G(\rho)|^{p'}|\nabla_g\rho|^{p} w(\rho) |H'(u)|^{p'}\diff v_g\right)^{p-1}},
		\end{equation}
		provided that there exists a neighborhood $U\subset \mathbb R$ of zero such that 
		\begin{equation}\label{assummptions}
			H'(s)\neq 0,\ \forall s\in U\setminus \{0\},\quad G(t)\neq 0,\ \forall t\in (0,\sup_\Omega\rho)\quad\mbox{and}\quad\mathcal H^n(|\nabla_g\rho|^{-1}(0))=0.
		\end{equation}
	\end{enumerate}
\end{theorem}
\begin{proof} (i) By the convexity of $\xi\mapsto|\xi|^{p}$, it turns out that
$$|\xi|^{p}\geq |\eta|^{p} + p  |\eta|^{p-2} (\xi-\eta)\eta=p|\eta|^{p-2}\xi\eta{-(p-1)}|\eta|^{p},\quad \forall \xi,\eta\in T^*\Omega. $$ Fix $u\in C_{0}^\infty(\Omega)$ arbitrarily and let $\xi:=\nabla_gu$ and $\eta:=v\nabla_g\rho$, where  
$$v=-|G(\rho)|^\frac{2-p}{p-1} |H'(u)|^\frac{2-p}{p-1}  G(\rho)  H'(u).$$
On the one hand, we have 
$$|\eta|^{p}=|G(\rho)|^\frac{p}{p-1} |\nabla_g\rho|^p |H'(u)|^\frac{p}{p-1}=|G(\rho)|^{p'} |\nabla_g\rho|^p |H'(u)|^{p'}.$$
On the other hand, the chain rule $\nabla_g H(u)=H'(u) \nabla_gu$ implies that
$$|\eta|^{p-2} \xi\eta=-G(\rho) |\nabla_g\rho|^{p-2}  \nabla_g\rho\nabla_g H(u).$$
Therefore, one has
$$
	|\nabla_gu|^{p}\ge- p G(\rho) |\nabla_g\rho|^{p-2}  \nabla_g\rho\nabla_g H(u){-(p-1)}|G(\rho)|^{p'}|\nabla_g\rho|^{p} |H'(u)|^{p'} .
$$
Multiplying by $w(\rho)>0$ and using \ref{item:C_0} we get
\begin{align}
	\int_\Omega	w(\rho)|\nabla_gu|^{p}  \diff v_g&\ge -p\int_\Omega	G(\rho) w(\rho) |\nabla_g\rho|^{p-2}  \nabla_g\rho\nabla_g H(u) \diff v_g\nonumber\\
	&\qquad{-(p-1)}\int_\Omega w(\rho) |G(\rho)|^{p'}|\nabla_g\rho|^{p} |H'(u)|^{p'} \diff v_g:=p I_H {-(p-1)} J_H.\label{eq:ihjh}
\end{align}
An integration by parts and the boundary conditions on $H$ yield
\begin{align*}
	I_H&=\int_\Omega   {\rm div}_g\left( G(\rho) w(\rho) |\nabla_g\rho|^{p-2}\nabla_g\rho\right)H(u)\diff v_g\\
	&=\int_\Omega  \left[ (G'(\rho) w(\rho)+ G(\rho) w'(\rho))  |\nabla_g\rho|^{p} +  G(\rho) w(\rho) \Delta_{g,p}\rho\right]H(u)\diff v_g.
\end{align*}
The above arguments imply inequality \eqref{additive}.

(ii) Recall the definitions of $I_H$ and $J_H$ from \eqref{eq:ihjh}; due to assumption \eqref{assummptions} we have that $J_H>0.$ Since \eqref{additive} is valid for $cH$ instead of $H$ for every $c\in \mathbb R$, it follows that
$$\int_\Omega w(\rho)|\nabla_gu|^{p}  \diff v_g\geq pI_{cH}{-(p-1)}J_{cH}=pcI_{H}{-(p-1)}|c|^{p'}J_{H}.$$ Note that ${p-1>0}$, thus we can maximize the right hand side of the latter expression with respect to $c$; the maximum is achieved for the value $c=J_H^{1-p}|I_H|^{p-2}I_H$, obtaining that
$$\int_\Omega	(w(\rho))|\nabla_gu|^{p}  \diff v_g\geq \frac{|I_H|^p}{J_H^{p-1}},$$
which is  inequality \eqref{multiplicative}.
\end{proof}

\begin{remark}\label{rem:equivalence}
	We notice that the additive and multiplicative forms can be deduced from each other. First,  relation \eqref{multiplicative} follows from \eqref{additive}, see  (ii). Second, we have that $$J_H^{1-p}|I_H|^p\geq pI_H{-(p-1)}J_H;$$ indeed, when $I_H\leq 0$, the latter inequality is trivial, while for $I_H> 0,$ 	
	  Young's inequality applies. %However, we shall see  that \eqref{additive} and \eqref{multiplicative} provide inequalities \eqref{eq:h} and \eqref{eq:up}, respectively.
\end{remark}

%{\it Proof of Theorem \ref{theorem-intro-main-no-weights}.}
%In Theorem \ref{theorem-weighted-main} we consider $w\equiv 1.$
% \hfill $\square$\\

% -(diff(y(x), x))-y(x)*(n-1)/x-y(x)^2-y(x)*(c*d*b*x^(c-1)/(a+b*x^c)-2*m/x) = C/x^2

Clearly, Theorem \ref{theorem-main-0} is a simple consequence of Theorem \ref{theorem-weighted-main} by choosing $w\equiv 1$ and the function $\rho$ such that $|\nabla_g \rho|=1$  $\diff v_g$-a.e.\ in $\Omega$. In particular, \ref{item:C_0} reduces to  \ref{item:C_0-no-w-0}. In the sequel, we prove a weighted version of Theorem \ref{Riccati-theorem-0}, which requires the weighted form of the $(p,\rho)$-Riccati pairs.
\begin{definition}\label{def:wRP}
	Let $(M,g)$ be a complete, non-compact $n$-dimensional Riemannian manifold, with  $n\geq 2$. Let  $\Omega\subseteq M$ be a domain, $p>1$, and $\rho\in W_{\rm loc}^{1,p}(\Omega)$ be a positive function with $|\nabla_g\rho|=1$  $\diff v_g$-a.e.\ in $\Omega$.  Let us fix the continuous functions $L,W\colon(0,\sup_\Omega\rho)\to (0,\infty)$ and the function $w\colon(0,\sup_\Omega\rho)\to (0,\infty)$ of class $C^1$.
	We say that the couple $(L,W)$ is a $(p,\rho,w)$\textit{-Riccati pair in} $(0,\sup_\Omega\rho)$ if there exists a  function $G\colon(0,\sup_\Omega\rho)\to \mathbb R$  such that 	
	\begin{enumerate}[label=\bf(R\arabic*)]
		\item\label{it:wr1} \ref{item:C_0} holds (from Theorem \ref{theorem-weighted-main});
		\item\label{it:wr2} $\Delta_g \rho\geq L(\rho)$ in the distributional sense in $\Omega,$ and $G\geq 0$ if $\mathcal H^n(\{x\in \Omega:\Delta_g \rho(x) > L(\rho(x))\})\neq 0.$
		\item\label{it:wr3} for every $t\in (0,\sup_\Omega\rho)$ one has
		\begin{equation}\label{Riccati pair-weighted}
			G'(t)+\left(\frac{w'(t)}{w(t)}+L(t)\right)G(t){-(p-1)}|G(t)|^{p'}\geq  W(t).
		\end{equation}
	\end{enumerate}
	A function $G$ satisfying the above conditions is  said to be $w$-\emph{admissible} for $(L,W)$.
\end{definition}

% As before, let $(M,g)$ be a complete  $n$-dimensional Riemannian manifold $(n\geq 2),$ $\Omega\subseteq M$ be a domain, and
% $\rho\colon\Omega\to (0,\infty)$ be a  function such that $|\nabla_g\rho|=1$  $\diff v_g$-a.e.\   in $\Omega$.
% %and ${\rm Im}\rho\subset (0,\infty)$ is open.
% Let us fix the continuous functions   $L,W\colon(0,\sup_\Omega\rho)\to (0,\infty)$ and the function $w\colon(0,\sup_\Omega\rho)\to (0,\infty)$ of class $C^1$.
% We say that the couple $(L,W)$ is a $(p,\rho,w)$\textit{-Riccati pair in} $(0,\sup_\Omega\rho)$ if there exists a  function $G\colon(0,\sup_\Omega\rho)\to \mathbb R$  such that 	
% \begin{itemize}
% 	\item[\textbf{(R1):}] \ref{item:C_0}  holds (from Theorem \ref{theorem-weighted-main});
% 	\item[\textbf{(R2):}] $\Delta_g \rho\geq L(\rho)$ in the distributional sense in $\Omega,$ and $G\geq 0$ if $\mathcal H^n(\{x\in \Omega:\Delta_g \rho(x) > L(\rho(x))\})\neq 0.$
% 	\item[\textbf{(R3):}] for every $t\in (0,\sup_\Omega\rho)$ one has
% 	\begin{equation}\label{Riccati pair-weighted}
% 		G'(t)+\left(\frac{w'(t)}{w(t)}+L(t)\right)G(t)+(1-p)|G(t)|^{p'}\geq  W(t).
% 	\end{equation}
% \end{itemize}
% %
%
%
%\begin{itemize}
%	\item \ref{item:C_0-no-w-0} holds (from Theorem \ref{theorem-main-0}), and
%	\item for every $t\in (0,\sup_\Omega\rho)$ one has
%	\begin{equation}\label{Riccati pair-0-0}
%		G'(t)+L(t)G(t)+(1-p)G(t)^{p'}\geq W(t).
%	\end{equation}
%\end{itemize}
%Clearly, in this setting, assumption \ref{item:C_0} considerably simplifies, due to the validity of the eikonal equation (i.e., $|\nabla_g\rho|=1$  $\diff v_g$-a.e.\   in $\Omega$).
In the unweighted case, i.e.\ $w\equiv 1$, the notion
of $(p,\rho,w)$-Riccati pair in $(0,\sup_\Omega\rho)$ reduces to the $(p,\rho)$-Riccati pair in $(0,\sup_\Omega\rho)$.
%\setlist[enumerate,1]{start=0}
%\begin{enumerate}[label=\textbf{\rm (\textbf{G})$_\rho$}]
%	\item\label{item:C_0-no-w}$:$
%	$ G(\rho) \in L_{\rm loc}^{p'}(\Omega)$ and   $   G'(\rho)\in L_{\rm loc}^{1}(\Omega)$.
%\end{enumerate}
A consequence of Theorem \ref{theorem-weighted-main} reads as follows, which is a weighted form of Theorem \ref{Riccati-theorem-0}:

\begin{theorem}\label{W-theorem}
	Let $(M,g)$ be a complete, non-compact $n$-dimensional Riemannian manifold, with  $n\geq 2$. Let  $\Omega\subseteq M$ be a domain, $p>1$, and $\rho\in W_{\rm loc}^{1,p}(\Omega)$ be a positive function with $|\nabla_g\rho|=1$  $\diff v_g$-a.e.\ in $\Omega$.
	Let $L,W\colon(0,\sup_\Omega\rho)\to (0,\infty)$ be continuous functions and $w\colon(0,\sup_\Omega\rho)\to (0,\infty)$ be of class $C^1$ such that $(L,W)$ is a $(p,\rho,w)$-Riccati pair in $(0,\sup_\Omega\rho)$.
	Then for every
	$u\in C_0^{\infty}(\Omega)$ one has \begin{equation}\label{W-Reduced}
		\int_\Omega	w(\rho)|\nabla_gu|^{p} \diff v_g \geq  \int_\Omega W(\rho)w(\rho)|u|^p \,   \diff v_g.
	\end{equation}
	%Let $(M,g)$ be a $n$-dimensional Riemannian manifold $(n\geq 2)$, let $\Omega\subseteq M$ be a domain,  and $x_0\in \Omega$ be fixed. Let $p>1$, $S=\{d_{x_0}(x):x\in \Omega\}$, and assume that
%	\begin{itemize}
%		\item $\Delta_g \rho\geq L(\rho)$ in the distributional sense in $\Omega;$
%		%	\item $G:S\to \mathbb R$ is a nonpositive differentiable function satisfying $({\rm G}),$
%		%	\item $W:S\to \mathbb R$ is a nonnegative continuous function, and
%		\item  $(L,W)$ is a $(p,\rho,w)$-Riccati pair in $(0,\sup_\Omega\rho)$.
%	\end{itemize}
\end{theorem}
\begin{proof} Let us choose in Theorem \ref{theorem-weighted-main}/(i) the function $H(s)=\frac{|s|^p}{p}$ for every $s\in \mathbb R$. Since $(L,W)$ is a $(p,\rho,w)$-Riccati pair in $(0,\sup_\Omega\rho)$, it admits a $w$-{admissible} function  $G\colon(0,\sup_\Omega\rho)\to \mathbb R$, thus conditions \ref{it:wr1}-\ref{it:wr3} are satisfied. Condition \ref{it:wr1} yields \ref{item:C_0}. Using \ref{it:wr3} we have
\begin{equation}\label{G-W-L}	G'(t)w(t)+\left({w'(t)}+L(t)w(t)\right)G(t){-(p-1)}|G(t)|^{p'}w(t)\geq W(t)w(t),\quad \forall t\in (0,\textstyle\sup_\Omega\rho).
\end{equation}
Since $|\nabla_g\rho|=1$  $\diff v_g$-a.e.\ in $\Omega$, assumption \ref{it:wr2} implies that $\Delta_{g,p}\rho=\Delta_g \rho\geq L(\rho)$ in the distributional sense in $\Omega$.
 
 On the one hand, if  $\Delta_g\rho =L(\rho)$ $\diff v_g$-a.e.\ in $\Omega$, inequality  \eqref{W-Reduced} follows directly from \eqref{additive} and \eqref{G-W-L}. On the other hand, if $\mathcal H^n(\{x\in \Omega:\Delta_g \rho(x) > L(\rho(x))\})\neq 0,$ by \ref{it:wr2} one has that $G$ is nonnegative. Accordingly, we have again by \eqref{additive} and \eqref{G-W-L} the estimates
 \begin{align*}
 	\int_\Omega	w(\rho) |\nabla_gu|^{p} \diff v_g &\geq  \int_\Omega  \left[ (G'(\rho) w(\rho)+ G(\rho) w'(\rho))  +  G(\rho) w(\rho) \Delta_{g}\rho\right]|u|^p\diff v_g \\
	&\qquad{-(p-1)}\int_\Omega |G(\rho)|^{p'}w(\rho) |u|^p\diff v_g\\
	&\geq\int_\Omega  \left[ (G'(\rho) w(\rho)+ G(\rho) w'(\rho))  +  G(\rho) w(\rho) L(\rho)\right]|u|^p\diff v_g\\
	&\qquad {-(p-1)}\int_\Omega |G(\rho)|^{p'}w(\rho) |u|^p\diff v_g\\
	&\geq \int_\Omega W(\rho)w(\rho)|u|^p \,   \diff v_g,
 \end{align*}
 which is precisely relation  \eqref{W-Reduced}.
\end{proof}
\begin{remark}\label{remark-sign} For the model space form ${\bf M}_\kappa^n$  and $\rho=d_{x_0}$ for some $x_0\in {\bf M}_\kappa^n$ fixed, it follows that $\Delta_g d_{x_0}= (n-1){\bf ct}_\kappa(d_{x_0})$ on ${\bf M}_\kappa^n$ except the point $x_0$ and its cut locus (which is empty for $\kappa\leq 0$, and the antipodal point of $x_0$ on the $\kappa$-sphere $\mathbb S_\kappa^n$ when $\kappa>0$). Therefore, in such cases we choose $L(t)=(n-1){\bf ct}_\kappa(t)$ (with the usual interval restriction when $\kappa>0$) and \textit{no sign restriction} should be imposed on the function $G$, see \ref{it:wr2} and the proof of Theorem \ref{W-theorem}.
	However, on generic Riemannian manifolds, the nonnegativeness of $G$ is indispensable in our arguments.
	\end{remark}
%  the validity of inequality  \eqref{W-Reduced}.
%
%
%   on account of  \eqref{G-W-L} it follows that
%
%
%
%
%$$
% 	\int_\Omega	(w(\rho))|\nabla_gu|^{p} \diff v_g \geq  \int_\Omega ((Ww)\circ \rho)|u|^p \,   \diff v_g\ \ {\rm for\ all}\ \  u\in C_0^{\infty}(\Omega),
%$$
%
%
%
% \begin{eqnarray*}
% 	-G' \circ d_{x_0}-\Delta_g d_{x_0}G\circ d_{x_0}+(1-p)|G\circ d_{x_0}|^{p'}&\geq& \\
% 	-G' \circ d_{x_0}-{(n-1)}{\bf ct}_\kappa(d_{x_0})G\circ d_{x_0}+(1-p)|G\circ d_{x_0}|^{p'}&\geq&W\circ d_{x_0}.
% \end{eqnarray*}
%  \hfill $\square$\\
In the sequel we establish the connection between Riccati and Bessel pairs. According to Duy, Lam and Lu \cite[Definition 1.1]{DuyLL}, if $p>1$ and $A,B\colon(0,R)\to \mathbb R$ are functions with $A$ being of class $C^1$, the couple $(A,B)$ is a $p$-\textit{Bessel pair in} $(0,R)$ if the ODE
\begin{equation}\label{Bessel-pair}
	\left(t^{n-1}A(t)|y'(t)|^{p-2}y'(t)\right)'+t^{n-1}B(t)|y(t)|^{p-2}y(t)=0
\end{equation}
has a positive solution in $(0,R)$; for the initial version  $(p=2)$, see  Ghoussoub and Moradifam \cite{GM-Annalen, GM-PNAS}.
We have the following
\begin{proposition}\label{Bessel-Riccati}
%	Let $(M,g)$ be a complete, non-compact $n$-dimensional Riemannian manifold with $n\geq 2$. Let $\Omega\subseteq M$ be a domain, $p>1$ and $\rho\in W_{\rm loc}^{1,p}(\Omega)$ be a positive function with $|\nabla_g\rho|=1$  $\diff v_g$-a.e.\ in $\Omega$. 
Let $R>0$ and  $w,W\colon(0,R)\to (0,\infty)$ be two potentials with $w$ of class $C^1$. The function $y>0$ is a solution of \eqref{Bessel-pair} on $(0,R)$ for the couple $(A,B)=(w,wW)$ if and only if 
	\begin{equation}\label{eq:g}
		G(t)=-\frac{|y'(t)|^{p-2}y'(t)}{y(t)^{p-1}}
	\end{equation}
	is a solution of 
	\begin{equation}\label{RP-ODE}
		G'(t)+\left(\frac{w'(t)}{w(t)}+\frac{n-1}{t}\right)G(t){-(p-1)}|G(t)|^{p'}=  W(t),
	\end{equation} 
	on $(0,R)$, that is precisely \eqref{Riccati pair-weighted} with equality and $L(t)=\frac{n-1}{t}$.
\end{proposition}
\begin{proof}
	Inserting $G$ from \eqref{eq:g} into \eqref{RP-ODE} a simple computation yields \eqref{Bessel-pair}; since all steps can be reversed, the two equations are equivalent.
\end{proof}

 \begin{remark}\label{bessel-pair-generalization} Given a Riemannian manifold $(M,g)$  with sectional curvature ${\bf K}\leq \kappa$ for some  $\kappa\in \mathbb R$, a more appropriate notion for the $p$-Bessel pair $(A,B)$ instead of \eqref{Bessel-pair} is
 \begin{equation}\label{Bessel-pair-modified}
 	\left({\bf s}^{n-1}_\kappa(t)A(t)|y'(t)|^{p-2}y'(t)\right)'+{\bf s}^{n-1}_\kappa(t)B(t)|y(t)|^{p-2}y(t)=0,\quad t\in (0,R).
 \end{equation}
Indeed, when $\kappa=0$, equation \eqref{Bessel-pair-modified} reduces to \eqref{Bessel-pair}, while for $\kappa\neq 0$, the density $\textbf{s}_\kappa$ encodes the  curvature and explains the choice of $L(t)=(n-1){\bf ct}_\kappa(t).$ This observation will be crucial in some functional inequalities in the forthcoming sections that will be obtained by means of Riccati pairs, see e.g.\ Cheng's comparison principle for the first eigenvalue (see Theorem \ref{Cheng-theorem}).
 \end{remark}
 
 We conclude the section by establishing the relationship between Bessel/Riccati pairs and the classical approach of supersolutions.

 \begin{remark}\label{rem:supersolution}
	By the theory of supersolutions, if a second order elliptic operator admits a positive supersolution, then the operator is positive in the sense of quadratic forms, see Davies~\cite[Theorem 4.2.1]{Daviesbook}. Accordingly, if \(\Omega\subseteq \mathbb{R}^n\) is a domain containing the origin, \(W\in L^1_\text{loc}(\Omega)\cap C(\Omega\setminus\{0\})\) and $\varphi\in C^2(\Omega\setminus\{0\})$ is positive such that 
	\begin{equation}\label{eq:supersolution}
		(-\Delta-W)\varphi(x)\ge 0,\quad \forall x\in \Omega\setminus\{0\},
	\end{equation}
	then the operator \((-\Delta-W)\ge 0\) is positive; more precisely, one has
	\[\int_\Omega |\nabla u|^2\diff x\ge \int_\Omega Wu^2\diff x,\quad \forall u\in C_0^\infty(\Omega).\]
	Observe that if \(\varphi(x)=y(|x|)\) for some function \(y\), then~\eqref{eq:supersolution} translates to 
	\[-y''(t)-\frac{n-1}{t}y'(t)-W(t)y(t)\ge0,\quad \forall t\in(0,\sup\{|x|: x\in \Omega\}).\]
	The equality case of the above relation perfectly agrees with equation~\eqref{Bessel-pair} for \(A\equiv1\), \(B=W\) and \(p=2\). We note that similar arguments apply in the weighted case and \(p>1\). In addition, by a suitable Laplace comparison, these arguments can be extended to Riemannian manifolds via relation~\eqref{Bessel-pair-modified}. In addition, the above observations extend to Riccati pairs by Proposition~\ref{Bessel-Riccati} and Remark~\ref{bessel-pair-generalization}.
 \end{remark}

%\begin{remark}
%	Uj Bessel egyenlet: ahol az \textbf{sc} fuggveny mas DENSITY kell ide!  van.
%	content...
%\end{remark}

\section{Applications I: Additive Hardy-type inequalities via Riccati pairs} \label{section-Hardy}
\subsection{Caccioppoli   inequalities}\label{Caccippoli-subsection}

The first simple consequence of Theorem \ref{theorem-weighted-main} is a Caccioppoli-type  inequality, proved by D'Ambrosio and Dipierro \cite[Theorems 2.1 \& 3.1; Corollary 2.3]{DDipierro}.

\begin{theorem}\label{theorem-Hardy-1}
	Let $(M,g)$ be a complete, non-compact $n$-dimensional Riemannian manifold, with  $n\geq 2$. Let  $\Omega\subseteq M$ be a domain, $p>1$, and $\rho\in W_{\rm loc}^{1,p}(\Omega)$ be a nonnegative function. If $\alpha\in \mathbb R$ such that $-(p-1-\alpha)\Delta_{g,p} \rho\geq 0$ in the distributional sense in $\Omega$, and $|\nabla_g\rho|^p\rho^{\alpha-p},\rho^\alpha\in L_{\rm loc}^1(\Omega)$,
	 then for every $u\in C_0^\infty(\Omega)$ one has
	\begin{equation}\label{first-Hardy}
		\int_\Omega	\rho^\alpha|\nabla_gu|^{p} \diff v_g \geq \left(\frac{|p-1-\alpha|}{p}\right)^p \int_\Omega	\rho^\alpha\frac{|u|^{p}}{\rho^p} |\nabla_g\rho|^p\diff v_g.
	\end{equation}
\end{theorem}
\begin{proof} 
	Without loss of generality, based on the  assumptions $|\nabla_g\rho|^p\rho^{\alpha-p},\rho^\alpha\in L_{\rm loc}^1(\Omega)$, we may assume that $\rho$ is positive; indeed, eventually we may replace $\rho$ by $\rho+\epsilon$ for some $\epsilon>0$, and then take the limit with respect to $\epsilon>0$, by using the above integrability assumptions. A similar reason is developed in \cite{DDipierro}.

	Let us choose in Theorem \ref{theorem-weighted-main} the functions
	$$w(t)=t^\alpha,\quad G(t)=-(p-1-\alpha)\frac{|p-1-\alpha|^{p-2}}{p^{p-1}}t^{1-p}\quad\mbox{and}\quad H(s)=\frac{|s|^p}{p},\quad\forall t>0,\ \forall s\in\mathbb{R}.$$ 
	Observe that the properties in \ref{item:C_0} are equivalent to $|\nabla_g\rho|^p\rho^{\alpha-p},\rho^\alpha\in L_{\rm loc}^1(\Omega)$. Using the assumption $-(p-1-\alpha)\Delta_{g,p} \rho\geq 0$ and Theorem \ref{theorem-weighted-main}/(i), we directly obtain inequality  \eqref{first-Hardy}.
\end{proof}
%If $u=0$, we have noting to prove; thus, let us assume that $u\neq 0.$    $t>0,$ it follows that
% $ G(\rho) \in L_{\rm loc}^{p'}(\Omega)$ and  $   G'(\rho)\in L_{\rm loc}^{1}(\Omega)$, thus (G) holds.
%
%Let us consider the functions
% Then, since  $-\Delta_g \rho\geq 0$ in $\Omega$ in the distributional sense, it follows that
%$$\int_\Omega H(u) \cdot(  G'(\rho) + G(\rho)\cdot \Delta_g\rho)\diff v_g=\left(\frac{p-1}{p}\right)^p\int_\Omega\frac{|u|^{p}}{\rho^p}(-1+\rho\Delta_g\rho)\diff v_g\leq -\left(\frac{p-1}{p}\right)^p\int_\Omega \frac{|u|^{p}}{\rho^p}\diff v_g\leq 0.$$
%Furthermore, we have
%$$|H'(u)|^{p'}(G(\rho))^{p'}=\left(\frac{p-1}{p}\right)^p\frac{|u|^{p}}{\rho^p}.$$
%Accordingly, both the additive form \eqref{additive-0} and multiplicative form \eqref{multiplicative-0} of Theorem \ref{theorem-intro-main-no-weights} imply
%\eqref{first-Hardy}.
% \hfill $\square$\\

Some comments are in order.

\begin{remark} \label{remark-Caccioppoli}
	(a) The above argument provides an alternative proof of the $L^p$-Caccioppoli inequality discussed by D'Ambrosio and Dipierro \cite[Theorems 2.1 \& 3.1; Corollary 2.3]{DDipierro}. The sharpness (and eventually the existence of nonzero extremizers) in \eqref{first-Hardy}  is a very delicate issue, even in the Euclidean setting, see e.g.\ Brezis and Marcus \cite{BM} and Filippas,  Maz'ya and  Tertikas \cite{FilMazTert1}.
	In \cite[Theorem 4.1]{DDipierro} a simple criterion is formulated for $\alpha =0$ under the assumptions of Theorem \ref{theorem-Hardy-1}; namely, the constant $((p-1)/p)^p$ is sharp, and (i) if $\rho^{1/p'}\in D^{1,p}(\Omega)$ then the function $\rho^{1/p'}$ is an extremizer, while (ii) if $\rho^{1/p'}\notin D^{1,p}(\Omega)$ and $p\geq 2$, then there is no extremal function. Here, $D^{1,p}(\Omega)$ stands for the completion of $C_0^\infty(\Omega)$ with respect to the Dirichlet norm $u\mapsto (\int_\Omega |\nabla_g u|^pdv_g)^{1/p}$. \medskip
%	In fact, they work with $\rho\in W_{\rm loc}^{1,p}(\Omega)$ which is assumed to be  $p$-superharmonic in distributional sense (i.e., $-\Delta_{g,p}  \rho\geq 0$)  and not necessarily verifying the eikonal equation as we have (i.e, $|\nabla_g\rho|=1$  $\diff v_g$-a.e.\ in $\Omega$).  Note however, that while D'Ambrosio and Dipierro \cite{DDipierro} use $p$-superharmonicity, we only need the $2$-superharmonicity  of $\rho$, even for establishing $L^p$-Caccioppoli inequalities for any $p>1$.
	
	(b) A typical example of $\rho$ verifying the assumptions of Theorem \ref{theorem-Hardy-1} is provided by $\rho(x)={\rm dist}_{g}(x,\partial \Omega)$, $x\in \Omega$, where $\Omega\subset M$ is a bounded domain. Note that $|\nabla_g\rho|=1$  $\diff v_g$-a.e.\ in $\Omega$ and
	 \ref{item:C_0} is  verified for the choices of $w$ and $G$ in Theorem \ref{theorem-Hardy-1}. Indeed, for any constant $\beta\in \mathbb R$ and compact set $K\subset \Omega$, it follows that  $\rho_1:=\max_{x\in K}\rho(x)\geq \min_{x\in K}\rho(x)=\rho_0>0$;  thus $$\int_K \rho^{\beta}\diff v_g\leq \max\left\{\rho_0^{\beta},\rho_1^{\beta}\right\}{\vol_g(K)}<\infty,$$ where ${\vol_g(K)}$ denotes the volume of $K\subset M.$  In particular, $\rho^{\alpha-p},\rho^\alpha\in L_{\rm loc}^1(\Omega)$. Moreover, by the eikonal equation (i.e., $|\nabla_g\rho|=1$  $\diff v_g$-a.e.\ in $\Omega$), one has  $\Delta_{g,p}\rho=\Delta_{g}\rho$ for every $p>1.$\medskip
	
	(c) Let $(M,g)$ be an $n$-dimensional complete Riemannian manifold with \textit{nonnegative sectional curvature} and $N$ be a minimal closed submanifold  of $M$. If $\rho(x)=\rho_N(x)={\rm dist}_g(x,N)$ then $\rho$ is superharmonic, i.e.,
	  $-\Delta_g  \rho\geq 0$ a.e.\ in $M,$ see Chen,   Leung and Zhao \cite[Theorem 2.8]{CLZ}.
	
	  Alternatively, if $(M,g)$ is a complete Riemannian manifold with \textit{nonnegative Ricci curvature} and $\Omega\subset M$ is a domain with non-empty piecewise smooth \textit{weakly mean convex} boundary, then $\rho(x)={\rm dist}_{g}(x,\partial \Omega)$ is superharmonic a.e.\ on $\Omega$, see Chen, Leung and Zhao \cite[Corollary 2.9]{CLZ}. In particular, in the Euclidean setting the above statement can be reversed as well, see Lewis, Li and Li \cite{LLL}. Related study on Finsler-type Minkowski spaces can be found in  Della Pietra,  di Blasio and Gavitone \cite{DeBlGav}.
	
%	  (d) The combination of (b) and (c) provides various applications of Theorem \ref{theorem-Hardy-1}.
\end{remark}

%Kimasoltam ide: sajatos alkalmazas:
%
%	Let $(M,g)$ be a complete, non-compact $n$-dimensional Riemannian manifold, $n\geq 2,$ and $\Omega\subset M$ be a domain such that $-\Delta_g \rho\geq 0$ in $\Omega$ in the distributional sense, where $\rho(x)={\rm dist}_{g}(x,\partial \Omega)$, $x\in \Omega$. It is clear that $\rho$ satisfies the eikonal equation  $|\nabla_g\rho|=1$ $\diff v_g$-a.e.\ in $\Omega.$
%	
%	A simple argument shows that $\rho^{-p}\in L_{\rm loc}^{1}(\Omega)$. Indeed, if

An immediate consequence of Theorem \ref{theorem-Hardy-1} is the estimate of the first Dirichlet eigenvalue of
the Riemannian $p$-Laplacian operator; for simplicity, we consider the unweighted case $(\alpha=0)$:

\begin{corollary}\label{corol-first}
		Let $(M,g)$ be a complete, non-compact $n$-dimensional Riemannian manifold with $n\geq 2$, $\Omega\subseteq M$ be a bounded domain, $p>1$, and $\rho(x)={\rm dist}_{g}(x,\partial \Omega)$ for every $x\in \Omega$.  If $-\Delta_g \rho\geq 0$ in the distributional sense  in $\Omega$, then the first Dirichlet eigenvalue of
	the Riemannian $p$-Laplacian can be estimated as
	$$\lambda_{1,g}(\Omega)=\inf_{u\in C_0^{\infty}(\Omega)\setminus \{0\}}\frac{\int_\Omega	|\nabla_gu|^{p} \diff v_g}{\int_\Omega	|u|^{p} \diff v_g}\geq \left(\frac{p-1}{p}\right)^p\frac{1}{R_\Omega^p},$$
	where $R_\Omega=\sup_{x\in \Omega}\rho(x)$ is the Riemannian-inradius of the domain $\Omega\subset M$.
\end{corollary}

%\begin{remark}
%(i) For $p=2$ by Della Pietra... in Minkowski spaces;
%	
%(ii)	$-\Delta_F \rho\geq 0$ is equivalent to the fact that $\Omega$ is weakly $F$-mean convex (Minkowski spaces): Giga-Pisante, etc.

%\end{remark}

Using the notation $R_\Omega=\sup_{x\in \Omega}\rho(x)$ from Corollary \ref{corol-first}, in the spirit of Brezis and Marcus \cite{BM}, and Barbatis, Filippas and Tertikas~\cite{BFT}, we provide an improvement of Theorem \ref{theorem-Hardy-1} with a suitable reminder term, whenever $1<p\leq 2$:

\begin{theorem}\label{theorem-Cacc-improv}
	Under the same assumptions as in Corollary \ref{corol-first}, if  $1<p\leq 2$, one has
	\begin{equation}\label{first-Hardy-improved}
		\int_\Omega	|\nabla_gu|^{p} \diff v_g \geq \left(\frac{p-1}{p}\right)^p \int_\Omega	\frac{|u|^{p}}{\rho^p}(1+R_p(\rho))\diff v_g
	\end{equation}
	for every $u\in C_0^\infty(\Omega)$, where
$$R_p(t)=\left(1+\log^{-1}\left(\frac{t}{eR_\Omega}\right)\right)^{p-2}\left(1+(2-p)\log^{-1}\left(\frac{t}{eR_\Omega}\right)+\log^{-2}\left(\frac{t}{eR_\Omega}\right)\right)-1 \geq 0,\quad t\in (0,R_\Omega).$$
In particular, for $p=2$, we have
\begin{equation}\label{first-Hardy-improved-2}
	\int_\Omega	|\nabla_gu|^{2} \diff v_g \geq \frac{1}{4} \int_\Omega	\frac{u^2}{\rho^2}\diff v_g+\frac{1}{4} \int_\Omega	\frac{u^2}{\rho^2}\log^{-2}\left(\frac{\rho}{eR_\Omega}\right)\diff v_g.
\end{equation}
\end{theorem}
\begin{proof} 
	In Theorem \ref{theorem-main-0} we  choose the functions
	$$G(t)=-\left(\frac{p-1}{p}\right)^{p-1}t^{1-p}\left(1+\log^{-1}\left(\frac{t}{eR_\Omega}\right)\right)^{p-1}\quad\mbox{and}\quad H(s)=\frac{|s|^p}{p},\quad\forall t\in(0,R_\Omega),\ \forall s\in\mathbb{R}.$$ Note  that $\sup_\Omega \rho=R_\Omega$ and observe that \ref{item:C_0-no-w-0} clearly holds, see Remark \ref{remark-Caccioppoli}/(b).
	%{\color{blue} I cannot deduce \eqref{first-Hardy-improved} from Theorem \ref{theorem-main-0}. By choosing such $G$, I find every term in \eqref{additive-no-weight-0} is negative because $G'<0$, $\Delta_g\rho<0$ and $(1-p)<0$. What am I missing? }

	%and $G(t)=0$ for $t\geq R_\Omega$.
	An elementary computation shows the validity of \eqref{first-Hardy-improved}. Note that for  the reminder term we have $R_p(t)\geq 0$ for every $t\in (0,R_\Omega)$ since
	$z:=\log^{-1}\left(\frac{t}{eR_\Omega}\right)\in (-1,0)$ and, due to the fact that $p\in (1,2],$ the function $h(z)=\left(1+z\right)^{p-2}\left(1+(2-p)z+z^2\right)-1$ is decreasing on $(-1,0)$ with $\lim_{z\to 0}h(z)=0$.
	%{\color{blue} I can understand this part.}
% \hfill $\square$
\end{proof}

%\begin{remark}
%??	Adimurthy ... $p=2$ [PAMS] Della Pietra,...
%\end{remark}

\subsection{Hardy  inequalities and their improvements on Cartan-Hadamard manifolds}\label{Hardy-subsection}

In the sequel, as we already noticed,  we shall see that our approach perfectly works on Cartan-Hadamard manifolds. Before doing this, the following $L^p$-Hardy inequality is stated on Riemannian manifolds with certain constraints, which has been first established by Kombe and \"Ozaydin \cite[Theorem 2.1]{KO-2009} for generic $p>1$; the initial version for $p=2$ is the celebrated result of Carron \cite[Theorem 1.4]{Carron}.

\begin{theorem}\label{them-hardy-1}
	Let $(M,g)$ be a complete, non-compact $n$-dimensional Riemannian manifold with $n\geq 2$. Let $\alpha\in \mathbb R$, $p>1$, and
	$\rho\colon M\to [0,\infty)$ be a function such that $\rho^{-1}(0)\subset M$ is compact,
	$|\nabla_g\rho|=1$ and $\Delta_g \rho\geq \frac{C}{\rho}$ in the distributional sense for some $C>0$ with the property that $C+1+\alpha>p>1$.  Then for every $u\in C_0^\infty(M\setminus \rho^{-1}(0))$ one has
	\begin{equation}\label{second-Hardy-C}
		\int_M\rho^\alpha	|\nabla_gu|^{p} \diff v_g \geq \left(\frac{C+1+\alpha-p}{p}\right)^p \int_M\rho^\alpha	\frac{|u|^{p}}{\rho^p} \diff v_g.
	\end{equation}
%	Moreover, the constant $\left(\frac{n-p}{p}\right)^p$ is optimal and never attained.
\end{theorem}
\begin{proof} Denote $q=C+1+\alpha-p>0$. In Theorem \ref{W-theorem} choose  $\Omega=M\setminus \rho^{-1}(0)$, as well as
	$$w(t)=t^\alpha,\quad L(t)=\frac{C}{t},\quad W(t)=\left(\frac{q}{pt}\right)^p\quad\mbox{and}\quad G(t)=\left(\frac{q}{pt}\right)^{p-1},\quad\forall t>0.$$ We obtain that $(L,W)$ is a $(p,\rho,w)$-Riccati pair in $(0,\infty)\supset (0,\sup_\Omega\rho)$ and $G$ is $w$-admissible for $(L,W)$. Indeed, $G>0$ satisfies $$G'(t)+\left(\frac{w'(t)}{w(t)}+L(t)\right)G(t){-(p-1)}|G(t)|^{p'}=  W(t),$$ the validity of \ref{item:C_0} reduces to $\rho^{\alpha-p},\rho^{\alpha}\in L_{\rm loc}^1(\Omega)$, which holds true since $\Omega=M\setminus \rho^{-1}(0)$. Thus, we may apply Theorem \ref{W-theorem}, which concludes the proof.
\end{proof}

\begin{remark} \label{remark-capacity}
	If the $p$-capacity of the compact set $\rho^{-1}(0)\subset M$ is zero, then \eqref{second-Hardy-C} is valid not only in $ C_0^\infty(M\setminus \rho^{-1}(0))$ but also in $C_0^\infty(M)$, see e.g.\ Carron \cite{Carron} and D'Ambrosio and Dipierro \cite{DDipierro}.   In particular, if $n\geq p$ and $\mathcal H^{n-p}(\rho^{-1}(0))<\infty$, 
%
%	the Hausdorff  dimension of $\rho^{-1}(0)\subset M$ is  less than $n-p,$
	 then the $p$-capacity of $\rho^{-1}(0)\subset M$ is zero, see Heinonen,  Kilpel\"ainen and Martio \cite{Heinonenetal}.
\end{remark}

A simple consequence of Theorem \ref{them-hardy-1} is the following weighted Hardy inequality.

\begin{corollary}\label{Hardy-alpha}
		Let $(M,g)$ be an  $n$-dimensional Cartan-Hadamard manifold.  Let $1<p<n +\alpha $ for some $\alpha\in \mathbb R$ and $x_0\in M;$  then for every $u\in C_0^\infty(M\setminus \{x_0\})$ one has
		\begin{equation}\label{second-Hardy}
			\int_M d_{x_0}^\alpha	|\nabla_gu|^{p} \diff v_g \geq \left(\frac{n+\alpha-p}{p}\right)^p \int_Md_{x_0}^\alpha	\frac{|u|^{p}}{d_{x_0}^p} \diff v_g.
		\end{equation}
		Moreover, the constant $\left(\frac{n+\alpha-p}{p}\right)^p$ is optimal.
\end{corollary}
\begin{proof}
 Let $\rho=d_{x_0}$ in Theorem \ref{them-hardy-1}; note that $|\nabla_g d_{x_0}|=1$ in $\Omega:=M\setminus \{x_0\}$ (see \eqref{eikonal}), and by the Laplace comparison, see Theorem \ref{comparison-theorem}/(I)/(i), it follows that $\Delta_g \rho\geq \frac{n-1}{\rho}$. In particular, we may choose $C=n-1$.
%{\color{blue} (Since Theorem \ref{comparison-theorem}/(I)/(i) needs the injective radius of $x_0$, whether we should mention that in Cartan-Hadamard manifold the $\text{inj}_{x_0}=\infty$?) }
%Since $\rho^{-1}(0)=\{x_0\}$, by Remark \ref{remark-capacity} one has that \eqref{second-Hardy}
%is valid on $C_0^\infty(M)$.

The sharpness of and the lack of extremal functions can be shown in the usual way, see e.g.\ Yang,  Su and Kong \cite[Theorem 3.1]{YSK}, Zhao \cite[Theorem 1.2]{Z} and Krist\'aly \cite[Theorem 4.1]{Kristaly-JMPA}.
% \hfill $\square$
% body
\end{proof}

\begin{remark}
	(a)	An alternative way to prove \eqref{second-Hardy} is to choose $$\rho=d_{x_0},\quad w(t)=t^\alpha,\quad G(t)=t^{1-p}\quad\mbox{and}\quad H(s)=\frac{|s|^p}{p},\quad \forall t>0,\ \forall s\in\mathbb{R},$$ 
	in the inequality \eqref{multiplicative} of Theorem \ref{theorem-weighted-main}. Indeed, by the Laplace comparison, see Theorem \ref{comparison-theorem}/(I)/(i), it turns out that $\Delta_g d_{x_0}\geq \frac{n-1}{d_{x_0}}$ in the distributional sense,  thus
	$$\int_M   \left[ (G'(\rho) w(\rho)+ G(\rho) w'(\rho))  |\nabla_g\rho|^{p} +  G(\rho) w(\rho) \Delta_{g,p}\rho\right]H(u)\diff v_g\geq\frac{n+\alpha-p}{p} \int_Md_{x_0}^\alpha	\frac{|u|^{p}}{d_{x_0}^p} \diff v_g$$
	and
	$(G(\rho))^{p'}w(\rho)|H'(u)|^{p'}=d_{x_0}^\alpha\frac{|u|^{p}}{d_{x_0}^p}.$
	It remains to apply \eqref{multiplicative}.\medskip
	
	(b) Inequality \eqref{second-Hardy} has been also established by D'Ambrosio and Dipierro \cite[Theorem 6.5]{DDipierro} for $\alpha=0$, and Zhao \cite[Theorem 1.2]{Z} (for special Finsler manifolds). Further results on Finsler manifolds  can be also found in Mester,  Peter and  Varga \cite{MPV}. \medskip
	%(ii)	We know that the constant $\left(\frac{n-p}{p}\right)^p$ in \eqref{second-Hardy} is sharp, but never attained; in the Finsler setting for $p=2$, see Huang, Krist\'aly and Zhao \cite{HKZ-TAMS}, while for the Riemannian and Euclidean setting (where ${\bf S=0}$), see D'Ambrosio and Dipierro \cite{D'ADipierro}, etc.
	
	(c) The proof of Theorem \ref{theorem-weighted-main} can be adapted to the compact case as well, where certain Hardy-type inequalities can be produced; such a result will be provided in the sequel. To do this, let $(M,g)$ be an $n$-dimensional compact Riemannian manifold, $\Omega\subset M$ be a domain and $x_0\in \Omega$.  Let  $w\colon(0,\sup_\Omega\rho)\to (0,\infty)$ and $G\colon(0,\sup_\Omega\rho)\to \mathbb R$  be  $C^1$  functions such that
	\begin{enumerate}[label=\textbf{\rm (\textbf{G})}$_{\rho,w}^{x_0}$]
		\item\label{compactcase}$:$
		$
			G(\rho)  w( \rho)^\frac{1}{p'} \in L_{\rm loc}^{p'}(\Omega\backslash\{x_0\})\ {\rm and} \  G'(\rho) w(\rho) ,   
			G(\rho)w'(\rho), w( \rho)\in L_{\rm loc}^{1}(\Omega\backslash\{x_0\}).
		$
		\end{enumerate}
	In particular, Theorem \ref{theorem-weighted-main}  holds for every $u\in C^\infty_0(\Omega\backslash\{x_0\})$. Moreover, if ${\bf Ric}\geq   0$ and the parameters $p$ and $\beta$ verify  $p\in (1,n)\cup (n,\infty)$,  $\beta<-n$ and $p+\beta>-n$, we have the  Hardy-type inequality
	\begin{align}\label{compachardyineql}
		\int_M |\nabla_g u|^p d_{x_0}^{p+\beta}{\rm d}v_g\geq \left(  \frac{|n+\beta|}{p} \right)^p\int_M |u|^p d_{x_0}^\beta {\rm d}v_g
	\end{align}
	for every $u\in C^\infty(M,x_0):=\{u\in C^\infty(M): u(x_0)=0 \}$.
The proof goes as follows. Let $\Omega=M$, $\rho=d_{x_0}$, as well as 
$$w(t)=t^{p+\beta},\quad  G(t)=-t^{1-p}\quad  \mbox{and}\quad  H_c(s)=c|s|^p, \quad \forall t>0,\ \forall  s\in \mathbb R ,$$ and for $c>0$ chosen later. Clearly, \ref{compactcase} holds. Thus, for any $u\in C^\infty_0(M\backslash\{x_0\})$,  by \eqref{additive} and Theorem \ref{comparison-theorem}/(II)/(i) we have that
		\begin{align*}
			\int_M |\nabla_g u|^p d_{x_0}^{p+\beta}{\rm d}v_g\geq \left( c|\beta+n|{-(p-1)}c^{p'}  \right)\int_M |u|^pd_{x_0}^\beta{\rm d}v_g.
		\end{align*}
		It is easy to check that the maximum of $c|\beta+n|{-(p-1)}c^{p'}$ in $c>0$ is $\left(\frac{|n+\beta|}{p}\right)^{p}$, 
%		\[
%		\max_{c>0}\left[\right]=\left.\left[c|\beta+n|+(1-p)c^{p'}\right]\right|_{c=\left(\frac{|n+\beta|}{p}\right)^{p-1}}=,
%		\]
		which implies that \eqref{compachardyineql} holds for every $u\in C^\infty_0(M\backslash\{x_0\})$.
	By the theory of capacity, see Meng, Wang and Zhao \cite[Lemma 3.2]{MWZ}, every $u\in C^\infty(M,x_0)$ can be approximated by a sequence of functions belonging to $ C^\infty_0(M\backslash\{x_0\})$  and hence, \eqref{compachardyineql} follows. The sharpness of  $\left(\frac{|n+\beta|}{p}\right)^{p}$ follows by  \cite{MWZ}.
\end{remark}

When $\alpha=0$ in Corollary \ref{Hardy-alpha}, the limit case $p=n$ does not provide any reasonable inequality similar to \eqref{second-Hardy}. In the next result we prove a parameter-depending Hardy inequality with logarithmic weights, which is valid also in the limit case $p=n$; similar results were established by Edmunds and Triebel \cite{EdmundsTriebel} in the Euclidean case, as well as by D'Ambrosio and Dipierro \cite[Theorem 6.5]{DDipierro}, Nguyen \cite{Nguyen-Royal-Hardy-Rellich}, and Zhao \cite[Theorem 1.3]{Z} on Riemannian/Finsler manifolds.

\begin{theorem}\label{them-elso}
	Let $(M,g)$ be an  $n$-dimensional Cartan-Hadamard manifold, $x_0\in M,$ and $\alpha,p\in \mathbb R$  such that  $1<p\leq n$ and $\alpha+1<p.$
 Then for every $u\in C_0^\infty(B_{x_0}(1)\setminus\{x_0\})$ one has
	\begin{equation}\label{limit-Hardy}
		\int_{B_{x_0}(1)}\log^\alpha(1/d_{x_0})	|\nabla_gu|^{p} \diff v_g \geq \left(\frac{p-\alpha-1}{p}\right)^p \int_{B_{x_0}(1)}\log^{\alpha-p}(1/d_{x_0})	\frac{|u|^{p}}{d_{x_0}^p } \diff v_g.
	\end{equation}
	Moreover, the constant $\left(\frac{p-\alpha-1}{p}\right)^p$ is optimal.
\end{theorem}
\begin{proof} Denote $f(t)=t\log(1/t)>0$ for every $t\in(0,1)$. In Theorem \ref{W-theorem}, let $\rho=d_{x_0}>0$ on the punctured ball $B_{x_0}(1)\setminus \{x_0\}$ and choose the functions $$w(t)=\log^{\alpha}(1/t),\quad L(t)=\frac{p-1}{t},\quad W_C(t)=C f(t)^{-p}> 0\quad\mbox{and}\quad G_c(t)=c f(t)^{1-p}>0,\quad \forall t\in(0,1),$$ and for some $C,c>0$, which will be determined later.  Condition \ref{item:C_0} holds on $B_{x_0}(1)\setminus\{x_0\}$. Since $p\le n$, the Laplace comparison, see Theorem \ref{comparison-theorem}/(I)/(i), implies $\Delta_{g,p}\rho\ge L(\rho)$. Moreover,
	%We intend to improve the Hardy inequality \eqref{second-Hardy} (for $p=2$) with the reminder term $u\mapsto \displaystyle\int_\Omega\frac{u^2}{d_{x_0}^2}\log^{-2}\left(\frac{d_{x_0}}{eD_\Omega}\right)\diff v_g$, i.e., for some $C>0$ one has
	\begin{align*}
		G_c'(t)+\left(\frac{w'(t)}{w(t)}+L(t)\right)G_c(t){-(p-1)}G_c(t)^{p'}&=\left(c(p-\alpha-1){-(p-1)}c^{p'}\right) f(t)^{-p}:=W_C(t),
	\end{align*}
	for every $t\in (0,1)$, 
	where $C=c(p-\alpha-1){-(p-1)}c^{p'}$. Thus, $(L,W_C)$ is a $(p,\rho,w)$-Riccati pair in $(0,1),$ and Theorem~\ref{W-theorem} implies that for every $u\in C_0^\infty(B_{x_0}(1)\setminus \{x_0\})$ one has
	$$\int_{B_{x_0}(1)}\log^\alpha(1/d_{x_0})	|\nabla_gu|^{p} \diff v_g \geq C \int_{B_{x_0}(1)}	\log^{\alpha-p}(1/d_{x_0})\frac{|u|^{p}}{d_{x_0}^p } \diff v_g.$$
	Since we want to maximize the value of $C=c(p-\alpha-1){-(p-1)}c^{p'}>0$ in $c>0$, it follows that the largest value is obtained for $c=((p-\alpha-1)/p)^{p-1}$ which provides  $C=((p-\alpha-1)/p)^p.$
	%The extension of the latter inequality from  $C_0^\infty(B_{x_0}(1)\setminus \{x_0\})$ to $C_0^\infty(B_{x_0}(1))$ is based again on  Remark \ref{remark-capacity},
	%Assumption (G) is verified due to Lemma  \ref{lemma-integrability} with the choices  $\Omega=B_{x_0}(1)$,  $s_0=\delta_0=1$, and
	%$f(s)=s^{-n}\log^{a-n}(e/s)$  for $s\in (0,1)$ and $a\in \{0,1\}$, respectively.
	The sharpness of the inequality is proved in Zhao \cite[Theorem 1.3]{Z}.
	% \hfill $\square$\\
\end{proof}

Since the Hardy inequality in Theorem \ref{Hardy-alpha} is sharp, but no nonzero extremals exist, one can expect improvements even in the Riemannian setting; for recent achievements, see e.g.\  Berchio,  D'Ambrosio, Ganguly and Grillo \cite{Berchio-Nonlinear},  Berchio,  Ganguly and Roychowdhury \cite{Berchio-JMAA},  Flynn,  Lam and Lu  \cite{Flynn-JFA}, Krist\'aly \cite{Kristaly-JMPA}.
% such a result is provided by Krist\'aly \cite[Theorem 4.2]{Kristaly-JMPA}  in the spirit of Adimurthi,  Chaudhuri and  Ramaswamy \cite{ACR}.

In the sequel, we provide an alternative approach  to establish improved Hardy inequalities on Cartan-Hadamard manifolds; for the simplicity of presentation,  we shall consider the case $p=2$ and $\alpha=0.$ The first such result `interpolates' between Corollary \ref{Hardy-alpha} and Theorem \ref{them-elso}.

\begin{theorem}\label{adimurth-theorem} Let $(M,g)$ be an  $n$-dimensional Cartan-Hadamard manifold, $n\geq 3,$
	and $\Omega\subset M$ be a bounded domain.  Let  $x_0\in \Omega$ and $D_\Omega=\sup_{x\in \Omega}{d_g(x_0,x)}$. Then for every $u\in C_0^\infty(\Omega)$ one has
	\begin{equation}\label{Adimurthi-Hardy}
		\int_\Omega	|\nabla_gu|^{2} \diff v_g \geq \frac{(n-2)^2}{4} \int_\Omega	\frac{u^2}{d_{x_0}^2}\diff v_g+\frac{1}{4} \int_\Omega	\log^{-2}\left(\frac{d_{x_0}}{eD_\Omega}\right)\frac{u^2}{d_{x_0}^2}\diff v_g.
	\end{equation}
\end{theorem}
\begin{proof}
	We are going to apply Theorem \ref{Riccati-theorem-0} for the set $\Omega\setminus \{x_0\}$, and the positive function $\rho=d_{x_0}$ on $\Omega\setminus \{x_0\}$ (with ${\rm Im}\rho\subseteq (0,D_\Omega)$), as well as  
	$$L(t)=\frac{n-1}{t}\quad\mbox{and}\quad W_C(t)=\frac{(n-2)^2}{4t^2}+\frac{C}{t^2}\log^{-2}\left(\frac{t}{eD_\Omega}\right),\quad\forall t\in (0,D_\Omega),$$ and for some $C>0$, that will be defined later. First, according to the Laplace comparison principle, see Theorem \ref{comparison-theorem}/(I)/(i), one has that $\Delta_g \rho\geq L(\rho)$ on $\Omega\setminus \{x_0\}$.
	Then, we intend to guarantee that $(L,W_C)$ is a $(2,\rho)$-Riccati pair in $(0,D_\Omega);$ to do this, we  are looking for a positive function $G_C$ which solves the
	Riccati-type ODE
	%We intend to improve the Hardy inequality \eqref{second-Hardy} (for $p=2$) with the reminder term $u\mapsto \displaystyle\int_\Omega\frac{u^2}{d_{x_0}^2}\log^{-2}\left(\frac{d_{x_0}}{eD_\Omega}\right)\diff v_g$, i.e., for some $C>0$ one has
	$$G_C'(t)+L(t)G_C(t)-G_C(t)^2= W_C(t),\quad t\in (0,D_\Omega).$$
	The   fundamental solution of the latter equation is given by
	%
	%
	%$$-(diff(y(x), x))-y(x)*(n-1)/x-y(x)^2 = (-2+n)^2/(4*x^2)+C/(x^2*log(17/x)^2)$$
	%
	$$G_C(t)=\frac{n-2}{2t}+\frac{1-\sqrt {1-4C}}{2t\log(\frac{eD_\Omega}{t})},\quad  t\in (0,D_\Omega).$$ The function $G_C$ is well-defined and positive whenever  $C\leq \frac{1}{4}$. Clearly, we choose the largest possible value for $C$, i.e., $C=\frac{1}{4}$.
	%, and we may extend $G_C$ to the whole $(0,\infty)$ in order to be of class $C^1$.
	With this choice of $C$, it turns out that  \ref{item:C_0-no-w-0}  is verified on $\Omega\setminus\{x_0\}$.
	% indeed, in Lemma \ref{lemma-integrability} we choose $s_0=\delta_0=D_\Omega$,  $f(s)=s^{-2}$ and then $f(s)=s^{-2}\log^{-a}(\frac{eD_\Omega}{s})$ with $a\in \{1,2\}$ for $s\in (0,s_0)$, respectively.
	Therefore, $(L,W_{1/4})$ is a $(2,\rho)$-Riccati pair in $(0,D_\Omega),$ and Theorem \ref{Riccati-theorem-0}  provides the proof of   \eqref{Adimurthi-Hardy} for  functions belonging to  $C_0^\infty(\Omega\setminus \{x_0\}).$ Since  $\rho^{-1}(0)=\{x_0\}$, by Remark \ref{remark-capacity} we have the validity of  \eqref{Adimurthi-Hardy} for every function
	in $C_0^\infty(\Omega)$, which concludes the proof.
\end{proof}

Another  improvement of the Hardy inequality in $\mathbb R^n$ is due to Brezis and V\'azquez \cite[Theorem 4.1]{BV};  more precisely,
if $\Omega\subset \mathbb R^n$ is a bounded domain $(n\geq 2)$, one has for every $u\in C_0^\infty(\Omega)$ that
\begin{equation}\label{Brezis-Vazquez-original}
	\int_\Omega	|\nabla u|^{2} \diff x \geq \frac{(n-2)^2}{4} \int_\Omega	\frac{u^2}{|x|^2}\diff x+{j_{0,1}^2}\left(\frac{\omega_n}{\vol(\Omega)}\right)^\frac{2}{n} \int_\Omega	u^2\diff x,
\end{equation}	
where $j_{0,1}\approx 2.4048$ is the first positive root of the Bessel function  $J_0$, and $\omega_n$ is the volume of the unit Euclidean ball. Inequality \eqref{Brezis-Vazquez-original} has been obtained by Schwarz symmetrization and an ingenious 1-dimensional analysis.

In the sequel, by using our approach based on Riccati pairs, we provide a Riemannian version of the result by Brezis and V\'azquez \cite{BV}, which sheds new light on the  appearance  of $j_{0,1}$ in \eqref{Brezis-Vazquez-original}; in fact, we prove a kind of interpolation inequality (see also Remark \ref{remark-brezis-vazquez}).  Let $j_{\nu,k}$ be the $k^{\rm th}$ positive root of the Bessel function $J_\nu$ of the first kind and degree $\nu\in \mathbb R.$

\begin{theorem}\label{Brezis-Vazquez-theorem} Let $(M,g)$ be an  $n$-dimensional Cartan-Hadamard manifold, $n\geq 2,$
	and $\Omega\subset M$ be a bounded domain.  Let  $x_0\in \Omega$ and $D_\Omega=\sup_{x\in \Omega}{d_g(x_0,x)}$. Then for every $\nu\in \left[0,\frac{n-2}{2}\right]$ and $u\in C_0^\infty(\Omega)$ one has 
	\begin{equation}\label{B-V-Hardy}
		\int_\Omega	|\nabla_gu|^{2} \diff v_g \geq \left(\frac{(n-2)^2}{4} -\nu^2\right)\int_\Omega	\frac{u^2}{d_{x_0}^2}\diff v_g+\frac{j_{\nu,1}^2}{D_\Omega^2} \int_\Omega	u^2\diff v_g.
	\end{equation}
\end{theorem}
\begin{proof}
 	We intend to apply Theorem \ref{Riccati-theorem-0} for the set $\Omega\setminus \{x_0\}$, for the function $\rho=d_{x_0}>0$ on $\Omega\setminus \{x_0\}$, as well as for  $$L(t)=\frac{n-1}{t}\quad\mbox{and}\quad W_C(t)=\frac{1}{t^2}\left(\frac{(n-2)^2}{4} -\nu^2\right)+C> 0,\quad \forall t\in (0,D_\Omega),$$ and for some $C>0$ which will be determined later.
	%
	%
	%
	%
	%
	%Let $\rho=d_{x_0}$ and for some $C>0$, we consider the function $W_C(t)=\frac{(n-2)^2}{4t^2}+C\geq 0$, $t>0$.
	%We intend to improve the Hardy inequality \eqref{second-Hardy} (for $p=2$) with the reminder term $u\mapsto \displaystyle\int_\Omega\frac{u^2}{d_{x_0}^2}\log^{-2}\left(\frac{d_{x_0}}{eD_\Omega}\right)\diff v_g$, i.e., for some $C>0$ one has
	We clearly have that $\Delta_g \rho\geq L(\rho)$ on $\Omega\setminus \{x_0\}$, while the fundamental solution of the  ODE
	$$G_C'(t)+L(t)G_C(t)-G_C(t)^2= W_C(t),\quad t\in (0,D_\Omega),$$
	is given by the function
	%
	%
	%$$-(diff(y(x), x))-y(x)*(n-1)/x-y(x)^2 = (-2+n)^2/(4*x^2)+C/(x^2*log(17/x)^2)$$
	%
	$$G_C(t)=\frac{n-2-2\nu}{2t}+\sqrt{C}\frac{J_{\nu+1}(\sqrt{C}t)}{J_\nu(\sqrt{C}t)},\quad  t\in (0,D_\Omega).$$
	%
	%$$G_C(t)=-\frac{2J_1(\sqrt{C}t)\sqrt{C}t+(n-2)J_0(\sqrt{C}t)}{2tJ_0(\sqrt{C}t)}$$
	Accordingly, $G_C$ is well-defined whenever
	$J_\nu(\sqrt{C}t)\neq 0$ for every $t\in (0,D_\Omega)$; thus, we need $\sqrt{C}D_\Omega\leq j_{\nu,1}$.
	Clearly, we choose the largest possible value $C:=j_{\nu,1}^2/D_\Omega^2$.

	%   for which  $G_C$ and $H(t)=\frac{t^2}{2}$ also verify the superposition property (HG).
		
	We have that $G_{{C}}$ is positive on $(0,D_\Omega)$.
	Indeed, since $n-2-2\nu>0$, it is enough to show that $J_{\nu+1}(t)/J_\nu(t)\geq 0$, for every $t\in (0,j_{\nu,1})$, which directly follows from the Mittag-Leffler representation, see Watson \cite[p. 498]{Watson},
	\begin{equation}\label{Mittag-Leffler}
		\frac{J_{\nu+1}(t)}{J_\nu(t)}=\sum_{k=1}^\infty \frac{2t}{j_{\nu,k}^2-t^2}> 0,\quad  t\in (0,j_{0,1}).
	\end{equation}
	%where $j_{0,k}$ is the $k^{\rm th}$ positive root of $J_0$.
	% Consequently, $G_C> 0$ on $(0,D_\Omega)$  with $C=j_{\nu,1}^2/D_\Omega^2$.

	A simple reasoning shows that \ref{item:C_0-no-w-0} holds on $\Omega\setminus\{x_0\}$ for the above choice of $G_{{C}}$, thus $(L,W_{{ C}})$ is a $(2,\rho)$-Riccati pair in $(0,D_\Omega).$
	%It remains to prove that assumption (G) is verified. To complete this, we first recall that $J_\nu(t)\approx \frac{1}{\Gamma(1+\nu)}\left(\frac{t}{2}\right)^\nu$ for $t\to 0^+$; in particular, $\frac{J_1(t)}{t^{a}J_0(t)}\in L^\infty(0,\delta)$ for every $\delta\in (0,j_{0,1})$ and $a\in \{0,1\}.$    By using Lemma  \ref{lemma-integrability}  with the choices   $s_0=\delta_0=D_\Omega$, and $f(s)=s^{-2}$, the above properties show that $G_C\circ d_{x_0}\in L_{\rm loc}^{2}(\Omega).$ By using the derivative formulas $J_0'(t)=-J_1(t)$ and $J_1'(t)=J_0(t)-\frac{J_1(t)}{t}$, $t>0$, it follows that $\left(\frac{J_1}{J_0}\right)'\in L^\infty(0,\delta)$ for every $\delta\in (0,j_{0,1})$. Therefore,  $G_C'\circ d_{x_0}\in L_{\rm loc}^{1}(\Omega),$ which concludes the validity of  assumption (G).
	%and then $f(s)=G_C'(s)$  for $s\in (0,s_0)$, respectively, by using also the fact that $J_0(t)\approx 1$ as $t\to 0^+$.
	Now, we are in the position to apply Theorem \ref{Riccati-theorem-0} obtaining
	$$\int_\Omega	|\nabla_gu|^{2} \diff v_g \geq  \int_\Omega u^2 W_{C}(\rho)   \diff v_g=\left(\frac{(n-2)^2}{4} -\nu^2\right) \int_\Omega	\frac{u^2}{d_{x_0}^2}\diff v_g+\frac{j_{\nu,1}^2}{D_\Omega^2} \int_\Omega	u^2\diff v_g,$$
	for all $u\in C_0^\infty(\Omega\setminus \{x_0\})$. Since  $\rho^{-1}(0)=\{x_0\}$, a similar argument as in  Remark \ref{remark-capacity} implies the validity of \eqref{B-V-Hardy} for every function
	in $C_0^\infty(\Omega)$. 
	% \hfill $\square$\\
\end{proof}

Several comments are in order.

\begin{remark} \label{remark-brezis-vazquez}
	(a) The  Brezis-V\'azquez inequality \eqref{Brezis-Vazquez-original} in the Euclidean space $\mathbb R^n$ contains the constant $\left(\omega_n/\vol(\Omega)\right)^{{2}/{n}}$ instead of ${D_\Omega^{-2}}$. 	This fact is explained by the \textit{Schwarz symmetrization} argument, rearranging the sets and functions into radially symmetric objects; indeed, under symmetrization,  the Dirichlet energy $u\mapsto \int_\Omega |\nabla u|^2\diff x$ non-increases (P\'olya-Szeg\H o inequality),
	the $L^2$-norm is preserved (Cavalieri principle), while the Hardy term $u\mapsto \int_\Omega u(x)^2/|x|^2\diff x$ non-decreases (Hardy-Littlewood-P\'olya  inequality). In this way, the problem reduces to a \textit{ball} (with volume $\vol(\Omega)$) and radially symmetric functions; thus a 1-dimensional analysis suffices.
	
	These arguments can be partially repeated in Cartan-Hadamard manifolds: P\'olya-Szeg\H o inequality  is valid whenever the \textit{Cartan-Hadamard conjecture}\footnote{The Cartan-Hadamard conjecture states that the classical Euclidean form of the isoperimetric inequality holds in  Cartan-Hadamard manifolds.} holds (e.g.\ in dimensions 2, 3 and 4), see Hebey \cite{Hebey}, and the $L^2$-norm is preserved; however,  nothing is known about the Hardy term $u\mapsto \int_\Omega u^2/d_{x_0}^2\diff v_g$ under symmetrization. \medskip
	
	(b) We notice that in the `complementary' geometric setting, i.e.\ on  Riemannian/Finsler manifolds with \textit{nonnegative Ricci curvature}, the Hardy term non-decreases under symmetrization, i.e., a  Hardy-Littlewood-P\'olya-type inequality holds, see Krist\'aly, Mester and Mezei \cite{KMM} and Krist\'aly and Szak\'al \cite{K-Sz}. In this way, Euclidean spaces can be viewed as threshold geometric structures  concerning the symmetric rearrangement of Hardy terms. \medskip
	
	(c) Note that if in Theorem \ref{Brezis-Vazquez-theorem} we consider the ball $\Omega=B_{x_0}(R)$ for some $R>0$, then by the Bishop-Gromov volume comparison principle, see Theorem \ref{comparison-theorem}/(I)/(ii), we have that
	$$D_\Omega^{-2}=R^{-2}\geq \left(\frac{\omega_n}{\vol_g(\Omega)}\right)^{{2}/{n}},$$
	which is in a perfect concordance with the original inequality \eqref{Brezis-Vazquez-original} containing the volume of the set. \medskip

(d) Inequality \eqref{B-V-Hardy} trivially interpolates between the inequality of Brezis and V\'azquez $(\nu=0)$ and the celebrated Faber-Krahn inequality $(\nu=\frac{n-2}{2}).$ In the latter case, we obtain a spectral estimate for the first Dirichlet eigenvalue on a domain $\Omega$. Further details are provided in the next section.
\end{remark}

%$W(t)=\left(\frac{n-1}{n}\right)^nt^{-n}\log^{-n}(e/t)$  and
%$G(t)=-\left(\frac{n-1}{n}\right)^{n-1} t^{1-n}\log^{1-n}(e/t)$ for $t\in (0,1)$.

%\textbf{Question: Berchio, Ganguly: 2017, JFA: C=1/4 ???}
%
%$-(diff(y(x), x))-y(x)*(n-1)*coth(x)-y(x)^2 \geq  (1/4)*(n-1)^2+C/x^2$

%Let us fix $\emptyset\neq S\subset \mathbb R$, the differentiable function   $G:S\to \mathbb R$ and the continuous function  $W:S\to \mathbb R$;  for a fixed $\kappa\geq 0$,  we say that $(G,W)$ is a \textit{$\kappa$-Riccati pair in $S$} if
%\begin{equation}\label{Riccati pair}
%	-G'(t)-{(n-1)}{\bf ct}_\kappa({t})G(t)+(1-p)|G(t)|^{p'}\geq W(t),\ t\in S.
%\end{equation}

\subsection{Spectral  estimates on Riemannian manifolds}\label{Spectral-estimate-subsection}
Let $(M,g)$ be an  $n$-dimensional Riemannian manifold $(n\geq 2),$ $\Omega\subset M$ be a  domain, and $p>1$.  The \textit{first Dirichlet eigenvalue} of $\Omega$ for the $p$-Laplace-Beltrami operator $-\Delta_{g,p}$ on $(M,g)$ is given by
$$	\lambda_{1,p}(\Omega):=\inf_{u\in C_0^\infty(\Omega)\setminus \{0\}} \frac{\int_\Omega	|\nabla_gu|^{p} \diff v_g }{\int_\Omega{|u|^p}\diff v_g}.$$

In this section we show the applicability of our method to obtain spectral gap estimates on Riemannian manifolds. First, we provide a new proof of \textit{Cheng's comparison principle}, see Cheng \cite{Cheng} (whose original proof is based on Barta's argument), then the \textit{Faber-Krahn inequality} on Cartan-Hadamard manifolds, see Chavel \cite{Chavel}, as well as \textit{McKean's spectral gap estimate} are proved. Then, we conclude the section with a short proof of the main result of  Carvalho and  Cavalcante \cite[Theorem 1.1]{Car-Cav-JMAA} concerning the  lower bound of $\lambda_{1,p}$ for the $p$-Laplacian on generic Riemannian manifolds.
%First, we provide a simple proof of the Faber-Krahn inequality on Cartan-Hadamard manifolds; .  Then, we provide simple proofs for the famous
%. Finally, we give a short proof
% of

\begin{theorem}\label{Cheng-theorem} {\rm (see Cheng \cite{Cheng})}
	Let $(M,g)$ be an $n$-dimensional Riemannian  manifold  with $n\geq 2$ and sectional curvature  ${\bf K}\leq\kappa$ for some $\kappa\in \mathbb R$. Supposing that $x_0\in M$ and $0<R<\min(\operatorname{inj}_{x_0},\pi/\sqrt{\kappa})$ $($with the usual convention, see Theorem~\ref{comparison-theorem}$),$ one has
	\begin{equation}\label{cheng-compar}
		\lambda_{1,2}(B_{x_0}(R))\geq \lambda_{1,2}(B_\kappa(R)),
	\end{equation}
where $B_\kappa(R)$ stands for a ball of radius $R$ in the model space form $\mathbf M_{\kappa}^n$.
\end{theorem}
\begin{proof}
	By standard compactness and symmetrization arguments, we known that  $\lambda_{1,2}(B_\kappa(R))$ is achieved by a positive, radially symmetric, non-increasing function  $v\in W_0^{1,2}(B_\kappa(R))$, i.e.,  $v_\kappa=v(d_{\kappa}(\textbf{0},\cdot))$ is the non-increasing profile function of $v$ on $(0,R)$, where $\textbf{0}$ is the center of the ball $B_\kappa(R)\subset \mathbf M_{\kappa}^n.$  By the Euler-Lagrange equation, it turns out that
	\begin{equation}\label{E-L}
		v_\kappa''(t)+(n-1){\bf ct}_\kappa(t)v_\kappa'(t)+\lambda_{1,2}(B_\kappa(R))v_\kappa(t)=0,\quad t\in (0,R),
	\end{equation}
	which can be written equivalently into the form
	\begin{equation}\label{G0-version}
		G_\kappa'(t)+L(t)G_\kappa(t)-G_\kappa^2(t)=W,\quad t\in (0,R),
	\end{equation}
	where $$L(t)=(n-1){\bf ct}_\kappa(t),\quad W\equiv\lambda_{1,2}(B_\kappa(R))\quad\mbox{and}\quad G_\kappa(t)=-\frac{v_\kappa'(t)}{v_\kappa(t)},\quad \forall t\in (0,R).$$ 
	Note  that $G_\kappa\ge 0$ on $(0,R)$, due to the fact that $v_\kappa$ is non-increasing and positive in $(0,R)$. If $\rho=d_{x_0}$ on $B_{x_0}(R)\setminus \{x_0\}$, the Laplace comparison, see Theorem \ref{comparison-theorem}/(I)/(i), implies that $\Delta_g \rho\geq L(\rho)$   on  $B_{x_0}(R)\setminus \{x_0\}$. Since \ref{item:C_0-no-w-0} also holds, thus $G_\kappa$ verifies all the requirements that $(L,W)$ is a $(2,\rho)$-Riccati pair in $(0,R)$. Accordingly, by Theorem \ref{Riccati-theorem-0} one has for every $u\in C_0^\infty(B_{x_0}(R)\setminus \{x_0\})$ that
	$$\int_{B_{x_0}(R)}	|\nabla_gu|^2 \diff v_g \geq  \lambda_{1,2}(B_\kappa(R))\int_{B_{x_0}(R)} u^2    \diff v_g.$$ On account of Remark \ref{remark-capacity}, the latter inequality is valid  for every function
	in $C_0^\infty(B_{x_0}(R))$, which concludes the proof of \eqref{cheng-compar}.
\end{proof}

\begin{remark}
	We notice that \eqref{E-L} is also equivalent to  $$\left({\bf s}^{n-1}_\kappa(t)v_0'(t)\right)'+\lambda_{1,2}(B_\kappa(R)){\bf s}^{n-1}_\kappa(t)v_0(t)=0,\ t\in (0,R),$$
	which is in a perfect concordance with  \eqref{Bessel-pair-modified} from  Remark \ref{bessel-pair-generalization}.
\end{remark}

In the case when the domain is not a ball, as in Cheng's result, a more powerful argument is needed; we shall consider only the case $\kappa=0,$ which corresponds to the famous Faber-Krahn inequality on Cartan-Hadamard manifolds:

\begin{theorem} \label{FK-theorem}	Let $(M,g)$ be an $n$-dimensional Cartan-Hadamard manifold $(n\geq 2)$ which satisfies the Cartan-Hadamard conjecture, and $\Omega\subset M$ be a bounded domain. Then we have
	\begin{equation}\label{FK-estimate}
		\lambda_{1,2}(\Omega)
		%		:=\inf_{u\in C_0^\infty(\Omega)\setminus \{0\}} \frac{\displaystyle\int_\Omega	|\nabla_gu|^{2} \diff v_g }{\displaystyle\int_\Omega{u^2}\diff v_g}
		\geq \lambda_{1,2}(\Omega^*)= j_{\frac{n}{2}-1,1}^2\left(\frac{\omega_n}{\vol_g(\Omega)}\right)^{{2}/{n}},
	\end{equation}
where $\Omega^*\subset \mathbb R^n$ is a ball   with $\vol(\Omega^*)=\vol_g(\Omega)$.
\end{theorem}
\begin{proof}
Let $u\in C_0^\infty(\Omega)\setminus\{0\}$ be arbitrarily fixed; without loss of generality, we may assume that $u\geq 0.$ Let $u^*\colon\Omega^*\to [0,\infty)$ be  the symmetric rearrangement of $u$ on the Euclidean ball $\Omega^*\subset \mathbb R^n$ (without loss of generality, we may assume that its center is the origin), see Hebey \cite{Hebey}. By layer cake representation and the validity of the Cartan-Hadamard conjecture on $(M,g)$, both the Cavalieri principle and P\'olya-Szeg\H o inequality hold, thus
$$\frac{\int_\Omega	|\nabla_gu|^{2} \diff v_g }{\int_\Omega{u^2}\diff v_g}\geq \frac{\int_{\Omega^*}	|\nabla u^*|^{2} \diff x }{\int_{\Omega^*}{(u^*)^2}\diff x}.$$
%({\color{blue}I really don't know how to prove this inequality. Maybe we show give more details})

Now, for the right hand side, we are falling into the setting of Theorem \ref{Cheng-theorem} in the case $\kappa=0$. Therefore, let $R=(\vol_g(\Omega)/\omega_n)^{{1}/{n}}$, $\Omega^*=B_0(R)$ and $W_C\equiv C=\lambda_{1,2}(B_0(R))$. Then
 then  equation \eqref{G0-version} for $\kappa=0$ has the solution
$$G_0(t)=\sqrt{C}\frac{J_\frac{n}{2}(\sqrt{C}t)}{J_{\frac{n}{2}-1}(\sqrt{C}t)},\quad  t\in (0,R).$$
%
%
%
%
%
%, in  Theorem \ref{Riccati-theorem-0} we choose the set $\Omega^*\setminus \{0\}$, the function $\rho=|x|>0$ on $\Omega^*\setminus \{0\}$,   $L(t)=\frac{n-1}{t}$ for every $t\in (0,R)$, and   $W_C\equiv C> 0$ for the largest possible value,
%
% in particular, if
%
%going to apply the Riccati pair technique. More precisely,  determined later. Note that $\Delta \rho= L(\rho)$ on $\Omega^*\setminus \{0\}$, and the fundamental solution of the  ODE
%$$G'(t)+L(t)G(t)-G(t)^2= W_C(t),\ t\in (0,R),$$
%is given by
%
%
%$$-(diff(y(x), x))-y(x)*(n-1)/x-y(x)^2 = (-2+n)^2/(4*x^2)+C/(x^2*log(17/x)^2)$$
%
%
%$$G_C(t)=-\frac{2J_1(\sqrt{C}t)\sqrt{C}t+(n-2)J_0(\sqrt{C}t)}{2tJ_0(\sqrt{C}t)}$$
A similar argument as  in the proof of Theorem \ref{Brezis-Vazquez-theorem} shows that the best choice is $$ C=\lambda_{1,2}(B_0(R))=\frac{j_{{\frac{n}{2}-1},1}^2}{R^2}.$$ Clearly,  $G_0$ is positive on $(0,R)$ and  \ref{item:C_0-no-w-0} holds; therefore,  $(L,W_{ C})$ is a $(2,\rho)$-Riccati pair in $(0,R)$. The rest similarly follows as in the proof of Theorem \ref{Cheng-theorem}.
%Accordingly,  \eqref{Riccati-Reduced-0} implies the estimate \eqref{FK-estimate}. Due to Remark \ref{remark-capacity}, the latter inequality is valid  for every function
%in $C_0^\infty(B_{x_0}(R))$, which concludes the proof of \eqref{cheng-compar}.
%Therefore, the function $G$ is well-defined whenever
%$J_{\frac{n}{2}-1}(\sqrt{C}t)\neq 0$ for every $t\in (0,R)$, which requires $\sqrt{C}R\leq j_{\frac{n}{2}-1,1}$;
%thus, the largest possible value is  $C=\frac{j_{{\frac{n}{2}-1},1}^2}{R^2}$. With this choice, $G$ is positive on $(0,R)$, which follows from the  Mittag-Leffler formula \cite[p.498]{Watson}, i.e.,
%$$\frac{J_{\frac{n}{2}}(t)}{J_{\frac{n}{2}-1}(t)}=\sum_{k=1}^\infty \frac{2t}{j_{\frac{n}{2}-1,k}^2-t^2}> 0,\  t\in \left(0,j_{\frac{n}{2}-1,1}\right).$$
%One can also show that  \ref{item:C_0-no-w} holds; therefore,  $(L,W_C)$ is a $\rho$-Riccati pair in $(0,R)$. Accordingly,  \eqref{W-Reduced} implies the estimate \eqref{FK-estimate}.
\end{proof}

The following result is  McKean's spectral gap estimate, established by McKean \cite{McKean} for $p=2$ by using fine properties of Jacobi fields; our argument is based on Riccati pairs.

\begin{theorem} \label{McKean-theorem}	Let $(M,g)$ be an  $n$-dimensional Cartan-Hadamard manifold $(n\geq 2),$  with sectional curvature ${\bf K}\leq  \kappa <0$. If $p>1$, then
	\begin{equation}\label{McKean-estimate}
		\lambda_{1,p}(M)
		%		:=\inf_{u\in C_0^\infty(M)\setminus \{0\}} \frac{\displaystyle\int_M	|\nabla_gu|^{p} \diff v_g }{\displaystyle\int_M{|u|^p}\diff v_g}
		\geq \left(\frac{n-1}{p}\sqrt{-\kappa}\right)^p.
	\end{equation}
\end{theorem}
\begin{proof} 
	Let $x_0\in M$ be arbitrarily fixed, and let  $\rho=d_{x_0}>0$ on $M\setminus \{x_0\}$, as well as 
	$$L\equiv(n-1)\sqrt{-\kappa},\quad W_C\equiv C \quad\mbox{and}\quad G_c\equiv c,$$
	for some $C,c>0$ that will be determined later. Clearly, \ref{item:C_0-no-w-0}  holds.

	The Laplace comparison, see Theorem \ref{comparison-theorem}/(I)/(i), yields 
	$$\Delta_g\rho\geq   {(n-1)}{\bf ct}_\kappa(\rho)=(n-1) \sqrt{-\kappa}\coth(\sqrt{-\kappa} \rho)\ge (n-1) \sqrt{-\kappa}= L(\rho).$$
	Therefore, by choosing $C=(n-1)c\sqrt{-\kappa}{-(p-1)}c^{p'}$ one has
	$$	G_c'(t)+L(t)G_c(t){-(p-1)}{G_c}(t)^{p'}\geq (n-1)c\sqrt{-\kappa}{-(p-1)}c^{p'}=C=W_C(t),\quad t>0,$$
	which proves that $(L,W_C)$ is a $(p,\rho)$-Riccati pair in $(0,\infty)$.\
	The maximum of  $C=(n-1)c\sqrt{-\kappa}{-(p-1)}c^{p'}$ in $c>0$ is obtained for $c=(\frac{n-1}{p}\sqrt{-\kappa})^{p-1}$, which is $C=(\frac{n-1}{p}\sqrt{-\kappa})^p$.  It remains to apply Theorem \ref{Riccati-theorem-0} to conclude the proof of \eqref{McKean-estimate}.
	% let
	%We observe that if $W\equiv\left(\frac{n-1}{p}\right)^p \kappa^\frac{p}{2}$ and $G\equiv-\left(\frac{n-1}{p}\right)^{p-1} \kappa^\frac{p-1}{2}$,  then $(W,G)$ forms a Riccati pair; indeed, , the inequality \eqref{Riccati pair} holds, i.e.,
\end{proof}

\begin{remark}
	(a) Another proof of \eqref{McKean-estimate} can be given by the multiplicative form from Theorem \ref{theorem-main-0}/(ii). Indeed, if $H(s)=\frac{|s|^p}{p}$ for $s\in \mathbb R$,  $G\equiv 1$ and $\rho=d_{x_0}$ for some $x_0\in M$, then the Laplace comparison, see Theorem \ref{comparison-theorem}/(I)/(i), implies that  $\Delta_{g}\rho\geq {(n-1)}{\bf ct}_\kappa(\rho)\geq (n-1)\sqrt{-\kappa}$ and \eqref{McKean-estimate} follows at once.\medskip
	
	(b) The constant  $(\frac{n-1}{p}\sqrt{-\kappa})^p$ in \eqref{McKean-estimate} is optimal for the $\kappa$-hyperbolic space $\mathbb H^n_{\kappa}$, i.e., we have that $$\lambda_{1,p}(\mathbb H^n_{\kappa})=\left(\frac{n-1}{p}\sqrt{-\kappa}\right)^p.$$ 
	
	(c)	Instead of \eqref{McKean-estimate}, we can prove an improved version of McKean's spectral gap estimate. Indeed, if $$L(t)=(n-1){\bf ct}_\kappa(t)\quad\mbox{and}\quad W(t)=\left(\frac{n-1}{p}\sqrt{-\kappa}\right)^p+\frac{(n-1)^p}{p^{p-1}} (-\kappa)^\frac{p}{2}(\coth(\sqrt{-\kappa}t)-1),\quad\forall t>0,$$ it turns out that $(L,W)$ is a $(p,\rho)$-Riccati pair in $(0,\infty)$,
	and
	$G\equiv(\frac{n-1}{p}\sqrt{-\kappa})^{p-1}$ is admissible for $(L,W)$. By Theorem \ref{Riccati-theorem-0} we have for every $u\in C_0^\infty(M)$ and $x_0\in M$ that
	$$\int_M	|\nabla_gu|^{p} \diff v_g\geq \left(\frac{n-1}{p}\sqrt{-\kappa}\right)^p\int_M	|u|^{p} \diff v_g+2\frac{(n-1)^p}{p^{p-1}} (-\kappa)^\frac{p}{2}\int_M \frac{	|u|^p}{e^{2\sqrt{-\kappa}d_{x_0}}-1} \diff v_g. $$
	%	The latter inequality proves, in particular, that the infimum in \eqref{McKean-estimate} cannot be attained.
	Another improvement of the McKean's spectral gap estimate will be provided in the next section.
\end{remark}

We conclude this section with providing a short proof of the main result of  Carvalho and  Cavalcante \cite[Theorem 1.1]{Car-Cav-JMAA}, valid on generic Riemannian manifolds:

\begin{theorem} \label{Car-Cav-theorem}	Let $(M,g)$ be a  Riemannian manifold, and $\Omega\subset M$ be a domain. Given $p>1$, we assume that there exists a function $\rho\colon\Omega\to \mathbb R$ such that $|\nabla_g\rho|\leq a$ and $\Delta_{g,p}\rho\geq b$ for some  $a,b>0$. Then
	\begin{equation}\label{CC-estimate}
		\lambda_{1,p}(\Omega)\geq
		%		:=\inf_{u\in C_0^\infty(M)\setminus \{0\}} \frac{\displaystyle\int_M	|\nabla_gu|^{p} \diff v_g }{\displaystyle\int_M{|u|^p}\diff v_g}
	\frac{b^p}{p^pa^{p(p-1)}}.
	\end{equation}
\end{theorem}
\begin{proof}
	We apply the additive form of Theorem \ref{theorem-main-0} with the choice $G\equiv c$ for some $c>0$ (which will be determined later) and $H(s)=\frac{|s|^p}{p}$ for every $s\in \mathbb R.$ Since the assumptions of Theorem \ref{theorem-main-0} are trivially verified, by \eqref{additive-no-weight-0} and the facts that $p>1$, $|\nabla_g\rho|\leq a$ and $\Delta_{g,p}\rho\geq b$, we obtain for every $u\in C_0^\infty(\Omega)$  that
	$$\int_\Omega	|\nabla_gu|^{p}  \diff v_g\geq \int_\Omega |u|^p\left(c\Delta_{g,p}\rho{-(p-1)}c^{p'}|\nabla_g\rho|^p\right)\diff v_g\geq \left(cb{-(p-1)}c^{p'}a^p\right)\int_\Omega |u|^p\diff v_g.$$
	Once we maximize the expression $cb{-(p-1)}c^{p'}a^p$ in $c>0$, the extremal point is  $c=(\frac{b}{pa^p})^{p-1}$, while the maximum is precisely the value $\frac{b^p}{p^pa^{p(p-1)}}$ in \eqref{CC-estimate}. The proof is complete.	
	%By using Riccati pairs, we reestablish in the sequel two important spectral estimates on Cartan-Hadamard manifold; the first is \textit{McKean's spectral gap estimate} when the sectional curvature verifies ${\bf K}\leq  -\kappa $ for some $\kappa> 0$, the second being the celebrated \textit{Faber-Krahn inequality}.
\end{proof}

 \subsection{Interpolation: Hardy inequality versus McKean spectral gap}\label{interpolation-mckean-hardy}

 The main result of this section is to prove an interpolation between the Hardy inequality and McKean's spectral gap, established first by Berchio,  Ganguly, Grillo and Pinchover \cite[Theorem 2.1]{Berchio-Royal}; although the authors stated their result on the hyperbolic space $\mathbb H^n_{-1}$, their approach also works on Cartan-Hadamard manifolds with sectional curvature ${\bf K}\leq  \kappa <0$. In the sequel,  we provide here a simple proof of the same result  by using our approach, based on Riccati pairs.

 \begin{theorem} \label{interpolation-theorem}	Let $(M,g)$ be an  $n$-dimensional Cartan-Hadamard manifold $(n\geq 3)$, having sectional curvature ${\bf K}\leq  \kappa <0$, and $x_0\in M$. Then, for every $\lambda\in [n-2,\frac{(n-1)^2}{4}]$ and $u\in C_0^\infty(M\setminus \{x_0\})$ one has
	\begin{align}
		\nonumber	\int_M |\nabla_g u|^2\diff v_g&\geq \lambda |\kappa|\int_M u^2\diff v_g+h_{n}^2(\lambda) \int_M \frac{u^2}{d_{x_0}^2}\diff v_g+|\kappa|\left(\frac{(n-2)^2}{4}-h_{n}^2(\lambda)\right) \int_M \frac{u^2}{\sinh^2(\sqrt{-\kappa}d_{x_0})}\diff v_g
		\\&\qquad + h_{n}(\lambda)\gamma_{n}(\lambda) \int_M \frac{{\bf D}_\kappa(d_{x_0})}{d_{x_0}^2}u^2\diff v_g,\label{new-estimate}
	\end{align}
where $\gamma_{n}(\lambda)=\sqrt{(n-1)^2-4\lambda}$ and $h_{n}(\lambda)=\frac{\gamma_{n}(\lambda)+1}{2}$.
 \end{theorem}
\begin{proof} Fix $\lambda\in [n-2,\frac{(n-1)^2}{4}]$. In Theorem \ref{Riccati-theorem-0} choose $\Omega=M\setminus \{x_0\}$, $\rho=d_{x_0}>0$, as well as
	\begin{align*}
		L(t)&={(n-1)}{\bf ct}_\kappa(t)=(n-1)\sqrt{-\kappa}\coth(\sqrt{-\kappa} t),\quad \forall t>0,\\
		W_\lambda(t)&=\lambda |\kappa| +\frac{h_{n}^2(\lambda)}{t^2}+\frac{|\kappa|}{\sinh^2(\sqrt{-\kappa}t)}\left(\frac{(n-2)^2}{4}-h_{n}^2(\lambda)\right)+h_{n}(\lambda)\gamma_{n}(\lambda) \frac{{\bf D}_\kappa(t)}{t^2},\quad \forall t>0.
	\end{align*}
	It is immediate that $W_\lambda$ is positive. The fundamental solution of the Riccati equation
	\begin{equation}\label{eq:ric:interpolation}
		G'(t)+L(t)G(t)-{G}(t)^2=W_\lambda(t),\quad \forall t>0,
	\end{equation}
	is given by
	$$G(t)=-\frac{h_n(\lambda)}{t}+\left(\frac{n-2}{2}+h_n(\lambda)\right){\bf ct}_\kappa(t),\quad \forall t>0.$$
	We easily observe that $G$ is positive on $(0,\infty)$ and \ref{item:C_0-no-w-0} holds. Since $\Delta_g \rho\geq L(\rho)$, see Theorem \ref{comparison-theorem}/(I)/(i), it turns out that $(L,W_\lambda)$ is a $(2,\rho)$-Riccati pair on $(0,\infty)$. Thus Theorem \ref{Riccati-theorem-0}  implies
	 inequality \eqref{new-estimate} for every
	 $u\in C_0^{\infty}(M\setminus \{x_0\})$ .
\end{proof}

 A direct consequence of Theorem \ref{interpolation-theorem} can be stated as follows for the two marginal values of $\lambda$, i.e., for  $\lambda=n-2$ a  Hardy improvement holds, while for $\lambda=\frac{(n-1)^2}{4}$ the McKean spectral gap is improved:

 \begin{corollary} \label{corollary-berchio-etal}	Let $(M,g)$ be an  $n$-dimensional Cartan-Hadamard manifold $(n\geq 3),$  with sectional curvature ${\bf K}\leq  \kappa <0$, and $x_0\in M$. Then, for every  $u\in C_0^\infty(M\setminus \{x_0\})$ one has
\begin{itemize}
	\item[{\rm (i)}] {\rm (Hardy improvement)}
		\begin{align}\label{new-estimate-0}
			\int_M |\nabla_g u|^2\diff v_g\geq\frac{(n-2)^2}{4} \int_M \frac{u^2}{d_{x_0}^2}\diff v_g
			 +(n-2) {|\kappa|}\int_M u^2\diff v_g+ \frac{(n-2)(n-3)}{2}\int_M \frac{{\bf D}_\kappa(d_{x_0})}{d_{x_0}^2}u^2\diff v_g.
	\end{align}
		\item[{\rm (ii)}] {\rm (McKean spectral gap improvement)}
			\begin{align}\label{new-estimate-1}
			\int_M |\nabla_g u|^2\diff v_g\geq\frac{(n-1)^2}{4}|\kappa| \int_M {u^2}\diff v_g + \frac{1}{4} \int_M \frac{u^2}{d_{x_0}^2}\diff v_g+|\kappa|\frac{(n-1)(n-3)}{4}\int_M \frac{u^2}{\sinh^2(\sqrt{-\kappa}d_{x_0})}\diff v_g.
		\end{align}
\end{itemize}
 \end{corollary}

%{\it Proof.}   \hfill $\square$
\begin{remark}\label{rem:criticality} (a)
	Inequality \eqref{new-estimate-0} appears in Berchio,  Ganguly, Grillo and Pinchover \cite[Corollary 4.3]{Berchio-Royal}. In addition,
	by this inequality and the fact that $$t^2\geq t\coth t-1,\quad\forall t>0,$$ we obtain the following (see Krist\'aly \cite[Theorem 4.1]{Kristaly-JMPA}):  for every  $u\in C_0^\infty(M\setminus \{x_0\})$ one has
	$$	\int_M |\nabla_g u|^2\diff v_g\geq\frac{(n-2)^2}{4} \int_M \frac{u^2}{d_{x_0}^2}\diff v_g
	+ \frac{(n-1)(n-2)}{2}\int_M \frac{{\bf D}_\kappa(d_{x_0})}{d_{x_0}^2}u^2\diff v_g.$$
	 Inequality \eqref{new-estimate-1} appears first in Akutagawa and Kumura \cite[Theorem 1.3/(6)]{Japanok}, which became a starting point of further studies, see e.g.\ Berchio,  Ganguly and  Grillo \cite[Theorem 2.1]{Berchio-JFA}, Berchio,  Ganguly,  Grillo and Pinchover \cite{Berchio-Royal}, and  Flynn,  Lam and  Lu \cite{Flynn-JFA}.

	 (b) We note that inequalities from Theorem~\ref{interpolation-theorem} and Corollary~\ref{corollary-berchio-etal} are known to be \textit{critical} on \(\mathbb{H}_{-1}^n\), i.e.,  the right hand sides of the inequalities cannot be improved with  weights that are strictly larger somewhere, see  Devyver, Fraas and Pinchover ~\cite[Definition~2.1]{DFP}. The criticality proofs are formulated using the supersolutions approach, however by Remarks~\ref{bessel-pair-generalization}~\&~\ref{rem:supersolution}, they can be adapted to Riccati pairs.

	To see this, let us consider~\eqref{new-estimate-1} on \(\mathbb{H}_{-1}^n\) and recall the definition of \(W_\lambda\) from the proof of Theorem~\ref{interpolation-theorem}. As we learned from~\cite{Berchio-JFA}, in order to prove the criticality, it is enough to show that the equation 
	\begin{equation}\label{eq:criticality}
		\left(-\Delta_g-W_\frac{(n-1)^2}{4}\right)\varphi(x)=0, \quad\forall x\in \mathbb{H}_{-1}^n\setminus\{x_0\},
	\end{equation}
	admits two solutions \(\varphi_\pm(x)=y_\pm(d_{x_0}(x))\) satisfying
	\(\lim_{r\to\infty}\frac{y_+(r)}{y_-(r)}=0.\) Moreover, by Remark~\ref{rem:supersolution} it is enough to consider the ODE~\eqref{Bessel-pair-modified} from Remark~\ref{bessel-pair-generalization}. By simple computation we obtain that the following linearly independent functions meet the above conditions:
	\[y_+(t)=\sqrt{t}\sinh(t)^\frac{2-n}{2}\quad\mbox{and}\quad y_-(t)=\sqrt{t}\log(t)\sinh(t)^\frac{2-n}{2},\quad \forall t>0.\]
	We note, however, that the above functions can be also obtained from two linearly independent solutions \(G_\pm\) of the Riccati ODE~\eqref{eq:ric:interpolation} via the mechanism \(G_\pm(t)=-\frac{y_\pm'(t)}{y_\pm(t)}\), see Proposition~\ref{Bessel-Riccati}. Using additional elements from the criticality theory, one can also adapt the criticality proof of the general inequality~\eqref{new-estimate}.\medskip

	(c) We notice that Berchio,  Ganguly and  Grillo ~\cite[Theorem 2.5]{Berchio-JFA} provided a more general version of inequality~\eqref{new-estimate-1} under a mild curvature assumption, which can be stated as follows. Let \((M,g)\) be an $n$-dimensional Riemannian manifold (\(n\ge 3\)), \(x_0\in M\) be a point such that \(\operatorname{Cut}_{x_0}=\emptyset\) (here, \(\operatorname{Cut}_{x_0}\) stands for the  cut locus of $x_0$), and the sectional curvature in the radial direction satisfies 
	\begin{equation}\label{eq:pointvise:curvature}
		{\bf K}_\text{rad}(d_{x_0}(x))\le -\frac{\psi''(d_{x_0}(x))}{\psi(d_{x_0}(x))},\quad\forall x\in M,
	\end{equation}
	where \(\psi\) is a positive, increasing \(C^2\) function with \(\psi(0)=\psi''(0)=0\), \(\psi'(0)=1\) and 
	\begin{equation}\label{eq:poscon}
		(n-2)\psi'(t)+(n-1)r\psi''(t)\ge 0,\quad\forall t>0.
	\end{equation}
	Then for every \(u\in C_0^\infty(M\setminus \{x_0\})\) one has 
	\begin{align}
		\int_M |\nabla_g u|^2\diff v_g&\geq\frac{(n-1)}{4}\int_M \left(2\frac{\psi''(d_{x_0})}{\psi(d_{x_0})}+(n-3)\frac{(\psi'(d_{x_0})^2-1)}{\psi(d_{x_0})^2}\right){u^2}\diff v_g\nonumber\\
		&\qquad + \frac{1}{4} \int_M \frac{u^2}{d_{x_0}^2}\diff v_g+\frac{(n-1)(n-3)}{4}\int_M \frac{u^2}{\psi^2(d_{x_0})}\diff v_g.\label{new-estimate-2}
	\end{align}
	Observe that for $\kappa<0$ and \(\psi(t)=\frac{\sinh(\sqrt{-\kappa}t)}{\sqrt{-\kappa}}\) the above inequality reduces to  \eqref{new-estimate-1}.

	As expected, inequality~\eqref{new-estimate-2} can be obtained via  Riccati pairs as well. Indeed, the curvature assumption~\eqref{eq:pointvise:curvature} implies the following refined Laplace comparison principle
	\[\Delta d_{x_0}\le (n-1)\frac{\psi'(d_{x_0})}{\psi(d_{x_0})},\]
	see Greene and Wu~\cite{Greene}. Thus inequality~\eqref{new-estimate-2} follows from Theorem~\ref{W-theorem} by choosing \[L(t)= (n-1)\frac{\psi'(t)}{\psi(t)}\quad\mbox{and}\quad G(t)=-\frac{1}{2t}+\frac{(n-1)}{2}\frac{\psi'(t)}{\psi(t)},\] 
	the positivity of \(G\) being guaranteed by the boundary conditions on \(\psi\) and relation~\eqref{eq:poscon}.  

	This example shows that Riccati pairs can be efficiently used to extend well-known Hardy inequalities to manifolds satisfying -- instead of a universal curvature bound -- the \emph{pointwise} curvature assumption~\eqref{eq:pointvise:curvature}.
	
\end{remark}

We conclude this subsection with a short proof of the following inequality, see Akutagawa and Kumura \cite[Theorem 1.3/(5)]{Japanok}.
\begin{theorem}\label{japanok-cikke-alapjan}
	Let $(M,g)$ be an  $n$-dimensional Cartan-Hadamard manifold with $n\geq 2$ and  sectional curvature ${\bf K}\leq  \kappa <0$. Let $x_0\in M$,  $R>0$ and $\Omega=M\setminus B_{x_0}(R)$. Then for every  and  $u\in C_0^\infty(\Omega)$ one has
	\begin{align*}\label{japan-estimate}
		\nonumber	\int_{\Omega}|\nabla_g u|^2\diff v_g&\geq \int_{\Omega}\left[\frac{(n-1)^2}{4}|\kappa|  +  \frac{1}{4\left(d_{x_0}-R+\frac{1}{(n-1){\bf ct}_\kappa(R)}\right)^2}+|\kappa|\frac{(n-1)(n-3)}{4\sinh^2(\sqrt{-\kappa}d_{x_0})}\right]{u^2} \diff v_g.
	\end{align*}
\end{theorem}
\begin{proof}
	The proof is similar to the one from Theorem \ref{interpolation-theorem}, except the choice of $W\colon(R,\infty)\to \mathbb R$ which is defined by
	$$W(t)=\frac{(n-1)^2}{4}|\kappa|  +  \frac{1}{4\left(t-R+\frac{1}{(n-1){\bf ct}_\kappa(R)}\right)^2}+|\kappa|\frac{(n-1)(n-3)}{4\sinh^2(\sqrt{-\kappa}t)},\quad t>R.$$
	One can check that the function
	$$G(t)=-\frac{1}{2\left(t-R+\frac{1}{(n-1){\bf ct}_\kappa(R)}\right)}+\frac{n-1}{2}{\bf ct}_\kappa(t),\quad t>R,$$
	is positive and verifies both \ref{item:C_0-no-w-0} and the ODE
	$$	G'(t)+L(t)G(t)-{G}(t)^2=W(t),\quad t>R,$$
	where $\rho=d_{x_0}$ and  $L(t)=(n-1){\bf ct}_\kappa(t)$;  thus $(L,W)$ is a $(2,\rho)$-Riccati pair on $(R,\infty)$, and the claim follows from Theorem
	\ref{Riccati-theorem-0}.
\end{proof}

\subsection{Ghoussoub-Moradifam-type inequalities}\label{gh-morad-subsection}

In this section we provide an alternative proof  of some inequalities established by Ghoussoub and Moradifam \cite[Theorem 2.12]{GM-Annalen}, see also \cite{GM-book},  where the weights are of the form $(a+b|x|^\alpha)^\beta/|x|^{2m}$ for some parameters. Possible extensions of these inequalities to Cartan-Hadamard manifolds will be also discussed (see Remark \ref{rem-gh-m} and Theorem \ref{0-Gh-Moradifam-theorem-0}), where some technical difficulties arise.

\begin{theorem}\label{Gh-Moradifam-theorem} {\rm (see \cite[Theorem 2.12]{GM-Annalen})}
	 Let $a,b>0$ and $\alpha,\beta,m\in  \mathbb R.$ The following inequalities hold$:$
	\begin{itemize}
		\item[{\rm (i)}] If $\alpha\beta>0$ and $m\leq \frac{n-2}{2}$, then for every $u\in C_0^\infty(\mathbb R^n)$ one has
		\begin{equation}\label{Gh-Moradifam-elso-egyenlet}
			\int_{\mathbb R^n} \frac{(a+b|x|^\alpha)^\beta}{|x|^{2m}}|\nabla u|^2\diff x\geq \left(\frac{n-2m-2}{2}\right)^2\int_{\mathbb R^n} \frac{(a+b|x|^\alpha)^\beta}{|x|^{2m+2}}u^2\diff x.
		\end{equation}
		\item[{\rm (ii)}] If $\alpha\beta<0$ and $2m-\alpha\beta\leq {n-2}$, then for every $u\in C_0^\infty(\mathbb R^n)$ one has
	\begin{equation}\label{Gh-Moradifam-masodik-egyenlet}
		\int_{\mathbb R^n} \frac{(a+b|x|^\alpha)^\beta}{|x|^{2m}}|\nabla u|^2\diff x\geq \left(\frac{n-2m+\alpha\beta-2}{2}\right)^2\int_{\mathbb R^n} \frac{(a+b|x|^\alpha)^\beta}{|x|^{2m+2}}u^2\diff x.
	\end{equation}
	\end{itemize}
\end{theorem}
\begin{proof}
	We first prove that (i) and (ii) are equivalent, i.e., they can be deduced from each other. Indeed, let us assume that (i) holds, i.e., by notation's convenience,  if
	$\tilde \alpha\tilde \beta>0$ and $\tilde m\leq \frac{n-2}{2}$, then for every $u\in C_0^\infty(\mathbb R^n)$ one has
	\begin{equation}\label{1-Gh-Moradifam-elso-egyenlet}
		\int_{\mathbb R^n} \frac{(b+a|x|^{\tilde \alpha})^{\tilde \beta}}{|x|^{2{\tilde m}}}|\nabla u|^2\diff x\geq \left(\frac{n-2\tilde m-2}{2}\right)^2\int_{\mathbb R^n} \frac{(b+a|x|^{\tilde\alpha})^\beta}{|x|^{2{\tilde m}+2}}u^2\diff x.
	\end{equation}
	Now, we fix $\alpha,\beta,m\in  \mathbb R$ which verify the assumptions from (ii), i.e., $\alpha\beta<0$ and $2m-\alpha\beta\leq {n-2}$. Let us choose $\tilde \alpha:=-\alpha$, $\tilde \beta:=\beta$ and $\tilde m:=m-\frac{\alpha\beta}{2}$; we observe that $$\tilde \alpha\tilde \beta=-\alpha\beta>0\quad\mbox{and}\quad\tilde m=m-\frac{\alpha\beta}{2}\leq  \frac{n-2+\alpha\beta}{2}-\frac{\alpha\beta}{2}=\frac{n-2}{2}.$$ Thus, we may apply \eqref{1-Gh-Moradifam-elso-egyenlet} with the latter choices, obtaining after a simple computation precisely
	\eqref{Gh-Moradifam-masodik-egyenlet}.  The converse can be performed similarly. Thus, it is enough to prove (i).

	Accordingly, we assume that $\alpha\beta>0$ and $m\leq \frac{n-2}{2}$. Define 
	$$w(t)=\frac{(a+bt^\alpha)^\beta}{t^{2m}},\quad L(t)=\frac{n-1}{t}\quad\mbox{and}\quad W_C(t)=\frac{C}{t^2},\quad \forall t>0,$$
	and for some $C>0$, that will be determined later. Consider the ODE
	\begin{equation}\label{Gh-Mor-eq}
		G'(t)+\left(\frac{w'(t)}{w(t)}+L(t)\right)G(t)-G(t)^2= W_C(t),\quad \forall t>0.
	\end{equation}
	If we introduce the numbers
	$$K_0:=n-2m-2\geq 0\quad\mbox{and}\quad K_1:=n-2m+\alpha\beta-2>0,$$  the fundamental solution of \eqref{Gh-Mor-eq} is given by
	\begin{equation}\label{fundamental-Gh-Mor}
	G(t)=\frac{K_0-\sqrt{K_0^2-4C}}{2t}\cdot\left(1-{\beta}\frac{bt^\alpha}{a}\mathcal I(t)\right),\quad t>0,
	\end{equation}
	where
	$$\mathcal I(t):=\frac{{_2F}_1\left(1+ \frac{\beta}{2}+ \frac{\sqrt{K_0^2-4C}-\sqrt{K_1^2-4C}}{2\alpha},1+ \frac{\beta}{2}+ \frac{\sqrt{K_0^2-4C}+\sqrt{K_1^2-4C}}{2\alpha};2+\frac{\sqrt{K_0^2-4C}}{2\alpha};-\frac{bt^\alpha}{a}\right)}{{_2F}_1\left( \frac{\beta}{2}+ \frac{\sqrt{K_0^2-4C}-\sqrt{K_1^2-4C}}{2\alpha}, \frac{\beta}{2}+ \frac{\sqrt{K_0^2-4C}+\sqrt{K_1^2-4C}}{2\alpha};1+\frac{\sqrt{K_0^2-4C}}{2\alpha};-\frac{bt^\alpha}{a}\right)},$$
	and ${_2F}_1$ stands for the regularized Gaussian hypergeometric function, i.e.,
	for $q,r,s\in \mathbb C$ ($s\notin \mathbb Z_-$),
	\begin{equation}\label{F-ertelmezes}
		{_2F}_1(q,r;s;z)=\frac{1}{\Gamma(1+s)} \sum_{k\geq 0}\frac{(q)_k(r)_k}{(s)_k}\frac{z^k}{k!},\quad \forall |z|<1,
	\end{equation}
	extended by analytic continuation elsewhere, while  $(q)_k=q(q+1)\ldots(q+k-1)=\frac{\Gamma(q+k)}{\Gamma(q)}$ denotes the Pochhammer symbol, where $k \in \mathbb N$.

	In order to have a well-defined  function $G$, we should guarantee first that $$C\leq \frac{1}{4}\min(K_0^2,K_1^2)=\left(\frac{n-2 m-2}{2}\right)^2.$$ Thus, in \eqref{fundamental-Gh-Mor} we put the best possible value, which is $C= \left(\frac{n-2 m-2}{2}\right)^2$. With this choice, the function $G$ reduces to
	\begin{equation}\label{Gh-Mor}
	G(t)=\frac{K_0}{2t}\cdot\left(1-{\beta}\frac{bt^\alpha}{a}\frac{{_2F}_1\left(1+ \frac{\beta}{2}- \frac{\sqrt{K_1^2-4C}}{2\alpha},1+ \frac{\beta}{2}+ \frac{\sqrt{K_1^2-4C}}{2\alpha};2;-\frac{bt^\alpha}{a}\right)}{{_2F}_1\left( \frac{\beta}{2}- \frac{\sqrt{K_1^2-4C}}{2\alpha}, \frac{\beta}{2}+ \frac{\sqrt{K_1^2-4C}}{2\alpha};
		1;-\frac{bt^\alpha}{a}\right)}\right),\ \ t>0.
	\end{equation}
	We  observe that the denominator in the latter expression does not vanish. To see this, let
	\begin{equation}\label{A-B-jeloles}
		A:=\frac{\beta}{2}\ \ {\rm  and}\ \  B:=\frac{\sqrt{K_1^2-4C}}{2\alpha}=\frac{\sqrt{\alpha\beta(\alpha\beta+2K_0)}}{2\alpha} ,
	\end{equation}
	and by using the
	Euler-Pfaff relation from \cite[rel. 15.8.1]{Olver}, we distinguish the following two cases:
	\begin{itemize}
		\item if $A>0$ (thus $\alpha>0$), then $B\geq A>0$ and for every $z>0$  we have 
		$${_2F}_1\left( A-B,A+B;1;-z\right)=(1+z)^{-A-B}{_2F}_1\left( 1-A+B,A+B;1;\frac{z}{1+z}\right)>0;$$
	\item 	if $A<0$ (thus $\alpha<0$), one has $B\leq A<0$, then for every $z>0$ it follows  that 
		$${_2F}_1\left( A-B,A+B;1;-z\right)=(1+z)^{B-A}{_2F}_1\left( A-B,1-A-B;1;\frac{z}{1+z}\right)>0.$$
	\end{itemize}
	Now, we are in the position to  apply Theorem \ref{W-theorem} for $p=2$, $\Omega=\mathbb R^n\setminus\{0\},$ $\rho(x)=|x|>0$ for $x\in \Omega$, as well as  $$w(t)=\frac{(a+bt^\alpha)^\beta}{t^{2m}},\quad L(t)=\frac{n-1}{t}\quad\mbox{and}\quad W(t)=\left(\frac{n-2 m-2}{2t}\right)^2,\quad\forall t>0.$$
	One can easily see that the integrability assumptions from \ref{item:C_0} are verified. By the above arguments it follows that $(L,W)$ is a
	$(2,\rho,w)$-Riccati pair in $(0,\infty)$. Therefore, by  \eqref{W-Reduced} we obtain the validity of  \eqref{Gh-Moradifam-elso-egyenlet}
	for every $u\in C_0^\infty(\mathbb R^n\setminus\{0\})$. It remains to extend \eqref{Gh-Moradifam-elso-egyenlet} to the space $ C_0^\infty(\mathbb R^n)$.
\end{proof}

 \begin{remark} \label{rem-gh-m}
 	One could expect a similar proof on Cartan-Hadamard manifolds as in Theorem \ref{Gh-Moradifam-theorem}. However, a subtle technical difficulty shows up which comes from the fact that -- in spite of several confirming numerical tests -- there is no evidence on the \textit{positiveness} of the function $G$ from \eqref{Gh-Mor} for the full range of parameters.   Note that in the Euclidean case such sign property is \textit{not} necessarily, as   $\Delta|x|=(n-1)/|x|$ for all $x\neq 0$, see Remark \ref{remark-sign} and condition \ref{it:wr2}.

% 	by multiplying the latter expression with $G(|x|)w(|x|)$, we still have a good estimate in \eqref{additive}. However, in generic Cartan-Hadamard manifolds only the Laplace comparison holds, i.e., $\Delta_g d_{x_0}\geq \frac{n-1}{d_{x_0}}$, which requires to have $G\geq 0$ in order to provide the correct estimate in \eqref{additive} and then to apply the ODE from \eqref{Gh-Mor-eq}.
 	
 	Closely related to the latter observation is the fact that $G$ contains the  expression $zf'(z)/f(z)$ for $f(z)={_2F}_1\left( a, b;
 	1;-z\right)$, $z>0$, which is used to characterize the \textit{order of starlikeness} (with respect to zero) of the function $f$, see e.g.\ K\"ustner \cite{Kustner}. Although one can find various results in the literature, to the best of our knowledge, there is no full characterization of the order of  starlikeness for Gaussian hypergeometric functions  with respect to the full spectrum of  parameters. However, under specific constraints on the parameters, we provide the following partial result on Cartan-Hadamard manifolds; for simplicity of presentation, we only focus to the inequality of type \eqref{Gh-Moradifam-elso-egyenlet}; by equivalence, it can be also formulated in terms of its `dual' \eqref{Gh-Moradifam-masodik-egyenlet}.
 \end{remark}

 \begin{theorem}\label{0-Gh-Moradifam-theorem-0}
 	Let $(M,g)$ be an  $n$-dimensional Cartan-Hadamard manifold $(n\geq 2)$, and $x_0\in M$ be a point. Let $a,b,\alpha,\beta>0,$ $m\in  \mathbb R$ with $m\leq \frac{n-2}{2}$ and $\alpha\beta+ \sqrt{\alpha\beta(\alpha\beta+2(n-2m-2))}\leq 2$. Then the following inequality holds for every $u\in C_0^\infty(M)$$:$
 		\begin{equation}\label{0-Gh-Moradifam-elso-egyenlet}
 			\int_M \frac{(a+bd_{x_0}^\alpha)^\beta}{d_{x_0}^{2m}}|\nabla_gu|^2\diff v_g\geq \left(\frac{n-2m-2}{2}\right)^2\int_M \frac{(a+bd_{x_0}^\alpha)^\beta}{d_{x_0}^{2m+2}}u^2\diff v_g.
 		\end{equation}
 \end{theorem}
\begin{proof}
	We apply Theorem \ref{W-theorem} for $p=2$, $\Omega=M\setminus\{x_0\},$  $\rho=d_{x_0}>0$ in $\Omega$, as well as, 
	$$w(t)=\frac{(a+bt^\alpha)^\beta}{t^{2m}},\quad L(t)=\frac{n-1}{t}\quad\mbox{and}\quad  W(t)=\left(\frac{n-2m-2}{2t}\right)^2,\quad\forall t>0.$$ 
	Note that the function $G$ from \eqref{Gh-Mor} verifies the ODE in \eqref{Gh-Mor-eq}. It is also clear that \ref{item:C_0} holds. Once we are able to prove that $G$ is positive on $(0,\infty)$, see also  Remark \ref{rem-gh-m}, it turns out that $(L,W)$ is a $(2,\rho,w)$-Riccati pair in $(0,\infty)$, and can apply Theorem \ref{W-theorem} concluding the proof of \eqref{0-Gh-Moradifam-elso-egyenlet}.

	To prove the positiveness of $G$ on $(0,\infty)$, it is enough to show that
	$$2Az\cdot \frac{{_2F}_1\left(1+ A-B,1+A+B;2;-z\right)}{{_2F}_1\left( A-B,A+B;1;-z\right)}\leq 1,\quad\forall z>0,$$
	where $A$ and $B$ are from \eqref{A-B-jeloles}; by our assumption, it follows that $B\geq A>0$.
	The definition of the hypergeometric function implies that the latter inequality is equivalent to
	$$\frac{2A}{A+B}\left(1-\frac{{_2F}_1\left(1+ A-B,A+B;1;-z\right)}{{_2F}_1\left( A-B,A+B;1;-z\right)}\right)\leq 1,\quad \forall z>0.$$
	In particular, since $B\geq A>0$, it is enough  to prove that
	$$\frac{{_2F}_1\left(1+ A-B,A+B;1;-z\right)}{{_2F}_1\left( A-B,A+B;1;-z\right)}\geq 0,\quad\forall  z>0.$$
	Since by assumption $\alpha\beta+ \sqrt{\alpha\beta(\alpha\beta+2(n-2m-2))}\leq 2$, we have that $|A\pm B|\leq 1$; in this case, we have the integral representation
	$$\frac{{_2F}_1\left(1+ A-B,A+B;1;-z\right)}{{_2F}_1\left( A-B,A+B;1;-z\right)}=\int_0^1 \frac{1}{1+tz}\diff \mu_0(t),\quad\forall z>0,$$
	where $\mu_0\colon[0,1]\to [0,1]$ is a non-decreasing function satisfying $\mu_0(1)-\mu_0(0)=1,$ see K\"ustner \cite{Kustner}. Clearly, by the latter representation our claim easily follows.
\end{proof}

\begin{remark}
As we already pointed out in Remark \ref{rem-gh-m}, numerical tests confirm the positiveness of $G$ for every $a,b>0$ and $\alpha,\beta,m\in  \mathbb R,$ whose proof requires some specific arguments from the theory of special functions; at this moment, such approach is not available. In particular, we expect to cancel the additional hypothesis  $\alpha\beta+ \sqrt{\alpha\beta(\alpha\beta+2(n-2m-2))}\leq 2$ from Theorem \ref{0-Gh-Moradifam-theorem-0}.
\end{remark}

\section{Applications II: Multiplicative Hardy-type inequalities}\label{section-uncertainty}
\subsection{Sharp uncertainty principles}\label{subsection-uncertainty-1}

Let $(M,g)$ be a complete, non-compact $n$-dimensional Riemannian manifold $(n\geq 2),$ and $x_0\in M$ be fixed. Given $p,\alpha\in \mathbb R$ such that $n>p>1$ and $-p+1<\alpha\leq 1$,
in this section we investigate the \textit{uncertainty principle}: for every  $u\in C_0^\infty(M)$,
%$$
%	\int_M	|\nabla_gu|^{p} \diff v_g\left(\int_Md_{x_0}^{p'\alpha}|u|^p\diff v_g\right)^{p-1}\geq\left(\frac{n+\alpha-1}{p}\right)^p \left(\int_M\left(1+\frac{n-1}{n+\alpha-1}{\bf D}_\kappa(d_{x_0})\right)d_{x_0}^{\alpha-1}|u|^p\diff v_g\right)^{p}. \eqno{(\textbf{UP})}
%$$
\begin{equation}\label{bizonytalansag}
	\left(\int_M	|\nabla_gu|^{p} \diff v_g\right)^\frac{1}{p}\left(\int_Md_{x_0}^{p'\alpha}|u|^p\diff v_g\right)^\frac{1}{p'}\geq\frac{n+\alpha-1}{p} \int_M\left(1+\frac{n-1}{n+\alpha-1}{\bf D}_\kappa(d_{x_0})\right)d_{x_0}^{\alpha-1}|u|^p\diff v_g.\tag{{\bf UP$_\kappa$}}
\end{equation}

We observe that \eqref{bizonytalansag} formally reduces to the:
\begin{itemize}
	\item \textit{Heisenberg-Pauli-Weyl uncertainty principle}, whenever $\alpha=1$; see Kombe and \"Ozaydin \cite{KO-2009, KO-2013} for $p=2$ and $\kappa=0$,  Krist\'aly \cite{Kristaly-JMPA} for $p=2$ and $\kappa\leq 0$, and Nguyen \cite{Nguyen-Royal-CKN} for generic $p>1$ and $\kappa\leq 0$;
	\item \textit{Hydrogen uncertainty principle}, whenever $\alpha=0$; see Cazacu,  Flynn and  Lam \cite{Cazacu-JFA} and Frank \cite{Frank} in $\mathbb R^n$ (thus $\kappa=0$);
	\item \textit{Hardy inequality} in the limit case $\alpha\to -p+1$, see Corollary \ref{Hardy-alpha}.
\end{itemize}

Our first result shows the validity of \eqref{bizonytalansag} on Cartan-Hadamard manifolds; namely, we have:

\begin{theorem} \label{Uncertainty-theorem}	Let $(M,g)$ be an $n$-dimensional Cartan-Hadamard manifold $(n\geq 2),$ such that ${\bf K}\leq \kappa\leq 0$. If $n>p>1$ and $-p+1<\alpha\leq 1$, then \eqref{bizonytalansag} holds and the constant $\frac{n+\alpha-1}{p}$ is sharp.
%	
%	
%	for every $u\in W^{1,p}(M)$ one has
%	\begin{equation}\label{uncertainity-ineq}
%		\int_M	F^*(Du)^{p} \diff v_g\left(\int_Md_{x_0}^{p'\alpha}|u|^p\diff v_g\right)^{p-1}\geq\left(\frac{n+\alpha-1}{p}\right)^p \left(\int_Md_{x_0}^{\alpha-1}|u|^p\diff v_g\right)^{p},
%	\end{equation}
%	where the constant $\left(\frac{n+\alpha-1}{p}\right)^p$ is sharp. In addition.
\end{theorem}
\begin{proof}
	Let us choose $$\rho=d_{x_0},\quad G(t)=t^\alpha\quad\mbox{and}\quad H(s)=\frac{|s|^p}{p},\quad \forall t>0,\ \forall s\in\mathbb R.$$ A simple computation shows that \ref{item:C_0} holds, due to our assumptions on $\alpha, p$ and $n$. Since $$\Delta_{g,p}\rho=\Delta_{g}\rho\geq (n-1){\bf ct}_\kappa(\rho),$$ see Theorem \ref{comparison-theorem}/(I)/(i), due to Theorem \ref{theorem-main-0}/(ii),  the inequality  \eqref{bizonytalansag} holds for every $u\in C_0^\infty(M\setminus \{x_0\}).$ It remains to extend \eqref{bizonytalansag} to the whole space $C_0^\infty(M).$

	We assume by contradiction that there exists $C>\frac{n+\alpha-1}{p}$ such that
	$$
	\left(\int_M	|\nabla_gu|^{p} \diff v_g\right)^\frac{1}{p}\left(\int_Md_{x_0}^{p'\alpha}|u|^p\diff v_g\right)^\frac{1}{p'}\geq C \int_M\left(1+\frac{n-1}{n+\alpha-1}{\bf D}_\kappa(d_{x_0})\right)d_{x_0}^{\alpha-1}|u|^p\diff v_g, \quad\forall u\in C_0^\infty(M).
	$$
	Since ${\bf D}_\kappa\geq 0$, the latter inequality implies that
	\begin{equation}\label{minek-kell}
		\left(\int_M	|\nabla_gu|^{p} \diff v_g\right)^\frac{1}{p}\left(\int_Md_{x_0}^{p'\alpha}|u|^p\diff v_g\right)^\frac{1}{p'}\geq C \int_Md_{x_0}^{\alpha-1}|u|^p\diff v_g, \quad\forall  u\in C_0^\infty(M).
	\end{equation}

	%Let us denote by $B(0,r)$ the
	%$n$-dimensional Euclidean ball of center $0$ and radius $r>0$.
	Fix $\varepsilon>0$ arbitrarily; then there exists a number $r>0$ and a  local diffeomorphism $\phi\colon B_0(r) \to M$ such that $\phi(0)=x_0$ and in the sense of bilinear forms, the    components $g_{ij}$
	of the metric $g$ satisfy on $\phi(B_0(r))$ the estimates
	\begin{equation}\label{two-sided}
		(1-\varepsilon)\delta_{ij}\leq g_{ij} \leq (1+\varepsilon)\delta_{ij},
	\end{equation} see Hebey \cite{Hebey}; here, as above, $B_0(r)$ is the $n$-dimensional Euclidean ball of center $0$ and radius $r>0$.

	Due to relations \eqref{minek-kell} and \eqref{two-sided}, for $0<\varepsilon\ll 1$ one can find $\tilde r>0$ with the property that
	for every $r\in (0,\tilde r)$ and $w\in
	C_0^\infty(B_0(r))$,
	the following inequality holds 
	\begin{equation}\label{c-hpw-uj-meg}
		\left(\int_{B_0(r)}	|\nabla w|^{p} \diff x\right)^\frac{1}{p}\left(\int_{B_0(r)}|x|^{p'\alpha}|w|^p\diff x\right)^\frac{1}{p'}\geq C\left(\frac{1-\varepsilon}{1+\varepsilon}\right)^{n+\alpha+1} \int_{B_0(r)}|x|^{\alpha-1}|w|^p\diff x.
	\end{equation}
	By assumption, $C>\frac{n+\alpha-1}{p}$, thus we may choose $0<\varepsilon\ll 1$ such that
	\begin{equation}\label{c-eps}
		C_\varepsilon:=C\left(\frac{1-\varepsilon}{1+\varepsilon}\right)^{n+\alpha+1}>\frac{n+\alpha-1}{p}.
	\end{equation}
	Let $u\in C_0^\infty(\mathbb R^n)$ be fixed arbitrarily and define the scaled function
	$w_\lambda(x)=u(\lambda x)$, $\lambda>0.$ For large enough $\lambda>0,$ it follows that
	$w_\lambda\in C_0^\infty(B_0(r))$; accordingly, we may use $w_\lambda$ as a test function in    (\ref{c-hpw-uj-meg}), obtaining
	$$		\left(\int_{B_0(r)}	|\nabla w_\lambda|^{p} \diff x\right)^\frac{1}{p}\left(\int_{B_0(r)}|x|^{p'\alpha}|w_\lambda|^p\diff x\right)^\frac{1}{p'}\geq C_\varepsilon \int_{B_0(r)}|x|^{\alpha-1}|w_\lambda|^p\diff x.$$
	Making the suitable change of variables, we have  the scaling properties
	\begin{align*}
		\int_{B_0(r)}|\nabla w_\lambda|^p{\text d}x&=\lambda^{p-n}\int_{\mathbb R^n}|\nabla u|^p{\text d}x,\\ 
		\int_{B_0(r)}|x|^{p'\alpha}|w_\lambda|^p\diff x&=\lambda^{-p'\alpha-n}\int_{\mathbb{R}^n}|x|^{p'\alpha}|u|^p\diff x,\\
		\int_{B_0(r)}|x|^{\alpha-1}|w_\lambda|^p\diff x&=\lambda^{-\alpha+1-n}\int_{\mathbb{R}^n}|x|^{\alpha-1}|w_\lambda|^p\diff x.
	\end{align*}
	In particular, by the scaling properties and the latter inequality, we obtain the $\lambda$-independent inequality
	$$ \left(\int_{\mathbb R^n}	|\nabla u|^{p} \diff x\right)^\frac{1}{p}\left(\int_{\mathbb R^n}|x|^{p'\alpha}|u|^p\diff x\right)^\frac{1}{p'}\geq C_\varepsilon \int_{\mathbb R^n}|x|^{\alpha-1}|u|^p\diff x.$$
	Since $-p+1<\alpha\leq 1,$ one can approximate  $$u(x)=e^{-\frac{|x|^{1+\frac{\alpha}{p-1}}}{p}},\quad \forall x\in \mathbb R^n,$$ by functions from $C_0^\infty(\mathbb R^n)$, and we may use it as a test function in the latter inequality. Taking polar coordinates and using the relation
	$$\int_0^\infty e^{-s^\beta}s^\gamma\diff s=\frac{1}{\beta}\Gamma\left(\frac{\gamma+1}{\beta}\right),\quad \forall\beta>0,\ \forall\gamma>-1,$$ we obtain that $\frac{n+\alpha-1}{p}\geq C_\varepsilon$, which contradicts \eqref{c-eps}, showing the sharpness of $\frac{n+\alpha-1}{p}$ in \eqref{bizonytalansag}.
\end{proof}

\begin{remark}
	If equality holds in \eqref{bizonytalansag} for some positive function $u\in W^{1,p}(M)$, then we should have also the equality $\Delta_{g}d_{x_0}= (n-1){\bf ct}_\kappa(d_{x_0})$, which implies that the manifold $(M,g)$ is isometric to the model space form ${\bf M}_\kappa^n$, see  Theorem \ref{comparison-theorem}; this rigidity result is known for $\kappa=0$  from Krist\'aly \cite{Kristaly-JMPA} (for $p=2$) and Nguyen \cite{Nguyen-Royal-CKN} (for $p>1$).
\end{remark}

The following result is a counterpart of Theorem \ref{Uncertainty-theorem}, providing the \textit{rigidity} of Riemannian manifolds with ${\bf Ric}\geq  \kappa (n-1)g$ for some $\kappa\leq 0$ supporting the uncertainty principle  \hyperref[bizonytalansag]{$({\textbf{UP}})_{\kappa}$}.\ Similar results have been obtained first by Krist\'aly \cite{Kristaly-JMPA} (for $p=2$ and $\alpha=1$) and then by Nguyen \cite{Nguyen-Royal-CKN} (for generic $p>1$), both of them considering only the case $\kappa=0$. Now, we have a more general result, valid for every $\kappa\leq 0:$

\begin{theorem} \label{Uncertainty-theorem-rigid}	Let $(M,g)$ be an $n$-dimensional complete, noncompact Riemannian manifold with $n\geq 2,$ such that ${\bf Ric}\geq  \kappa (n-1)g$ for some $\kappa\leq 0$. Suppose that  \eqref{bizonytalansag} holds for some $x_0\in M$ and the parameters $\alpha,p,n$ verify either $-p+1<\alpha\le 1< p< n$, when $\kappa=0$, or $0<\alpha\leq 1<p<n$, when  $\kappa<0$. Then $(M,g)$ is isometric to the model space form ${\bf M}_\kappa^n$.
\end{theorem}
\begin{proof}
	For simplicity of notations, let $\gamma=1+\frac{\alpha}{p-1}>0$.
	Let $u_0\colon M\to (0,\infty)$ be defined by $u_0=e^{-\frac{d_{x_0}^\gamma}{p}}$. Clearly, we may approximate $u_0$ by elements of $C_0^\infty(M)$; moreover, by using the eikonal equation \eqref{eikonal}, it turns out that $$\int_M	|\nabla_g u_0|^{p} \diff v_g=\left(\frac{\gamma}{p}\right)^p\int_Md_{x_0}^{p'\alpha}u_0^p	 \diff v_g.$$
	We first claim that the latter integrals are well defined. To see this, we observe that $\kappa=0$ implies $\gamma>0$ and $V_\kappa(R)=\omega_n R^n$; while $\kappa<0$ implies $\gamma>1$ and $$V_\kappa(R)\sim e^{(n-1)\sqrt{|\kappa}| R}\quad\mbox{as }R\to \infty,$$ up to a multiplicative constant. Now, by the layer cake representation and the volume comparison, see Theorem \ref{comparison-theorem}, we conclude the claim.
	%Since $\gamma=1+\frac{\alpha}{p-1}>1$, , see Theorem \ref{comparison-theorem}/(II)/(ii), it follows that the latter integrals are finite.

	By using $u_0$ as a test function in \eqref{bizonytalansag}, it follows that
	\begin{equation}\label{ineq-kappa}
		{\gamma}\int_Md_{x_0}^{p'\alpha}u_0^p	 \diff v_g\geq (n+\alpha-1) \int_M\left(1+\frac{n-1}{n+\alpha-1}{\bf D}_\kappa(d_{x_0})\right)d_{x_0}^{\alpha-1}u_0^p\diff v_g.
	\end{equation}
	By the Laplace comparison, see Theorem \ref{comparison-theorem}/(II)/(i), we have that $\Delta_g d_{x_0}\leq (n-1){\bf ct}_\kappa(d_{x_0})$. Therefore, by \eqref{ineq-kappa}, an integration by parts and the eikonal equation \eqref{eikonal} imply that
	\begin{align*}
	{\gamma}\int_Md_{x_0}^{p'\alpha}e^{-{d_{x_0}^\gamma}}	 \diff v_g&\geq \int_M\left(n+\alpha-1+{(n-1)}{\bf D}_\kappa(d_{x_0})\right)d_{x_0}^{\alpha-1}e^{-{d_{x_0}^\gamma}}\diff v_g\\
	&=\int_M\left(\alpha+{(n-1)}d_{x_0}{\bf ct}_\kappa(d_{x_0})\right)d_{x_0}^{\alpha-1}e^{-{d_{x_0}^\gamma}}\diff v_g\\
	&\geq \int_M\left(\alpha+d_{x_0}\Delta_g d_{x_0}\right)d_{x_0}^{\alpha-1}e^{-{d_{x_0}^\gamma}}\diff v_g\\
	&=\alpha\int_M d_{x_0}^{\alpha-1}e^{-{d_{x_0}^\gamma}}\diff v_g +\int_M d_{x_0}^{\alpha}e^{-{d_{x_0}^\gamma}}\Delta_g d_{x_0}\diff v_g\\&=\gamma\int_M d_{x_0}^{\alpha+\gamma-1}e^{-{d_{x_0}^\gamma}}\diff v_g
	\\
	&={\gamma}\int_Md_{x_0}^{p'\alpha}e^{-{d_{x_0}^\gamma}}	 \diff v_g,\end{align*}
	where we used the fact that $\alpha+\gamma-1=p'\alpha.$
	Since the two marginal terms are equal (and finite), we should have equality in each steps. In particular, we have that $\Delta_g d_{x_0}= (n-1){\bf ct}_\kappa (d_{x_0})$ on $M\setminus \{x_0\}$, thus by
	Theorem \ref{comparison-theorem} it follows that  $(M,g)$ is isometric to the model space form ${\bf M}_\kappa^n$. 
\end{proof}

\begin{remark}
	The proof of Theorem \ref{Uncertainty-theorem-rigid}
is much simpler than the original rigidity argument performed in Krist\'aly \cite{Kristaly-JMPA} and Nguyen \cite{Nguyen-Royal-CKN}, where  certain ordinary differential inequalities/equations are  compared.
\end{remark}

\begin{remark}
	By using additive-type inequalities we can produce  multiplicative inequalities in spirit of Berchio,  Ganguly,  Grillo and Pinchover \cite[Section  3]{Berchio-Royal}. For example, 	inequality \eqref{new-estimate-0}  combined with Schwarz inequality imply (in the same geometric context as in Corollary \ref{corollary-berchio-etal})  for every $u\in C_0^\infty(M)$ that
	$$
	\left(\int_M |\nabla_g u|^2\diff v_g-(n-2) {|\kappa|}\int_M u^2\diff v_g\right) \left(\int_M d_{x_0}^2 u^2\diff v_g\right) \geq\frac{(n-2)^2}{4} \left(\int_M {u^2}\diff v_g\right)^2
	.$$
	In the special case when $M=\mathbb H_{-1}^n$, the latter inequality reduces to \cite[Corollary 3.1]{Berchio-Royal}. In a similar way, further inequalities can be obtained via the results from \S \ref{interpolation-mckean-hardy}. 
\end{remark}

\subsection{Caffarelli-Kohn-Nirenberg inequalities}\label{CKN-subsection}
In this section we consider a version of the \emph{Caffarelli-Kohn-Nirenberg inequality}, which easily follows from the multiplicative form of Theorem \ref{theorem-weighted-main}.
\begin{theorem}\label{CKN-th}
	Let $(M,g)$ be an $n$-dimensional Cartan-Hadamard manifold, with $n\geq 2$ such that ${\bf K}\leq -\kappa$ for some $\kappa\geq 0$. Let $x_0\in M$ be fixed.
	Suppose that $p,\alpha,r\in \mathbb R$ satisfy
	$$r>p>1,\quad \alpha+p>1\quad\mbox{and}\quad p(n+\alpha-1)>r(n-p)>0,$$
	then for every  $u\in C_0^\infty(M)$ one has
	\begin{equation*}
	\left(\int_M	|\nabla_gu|^{p} \diff v_g\right)^\frac{1}{p}\left(\int_Md_{x_0}^{p'\alpha}|u|^{p'(r-1)}\diff v_g\right)^\frac{1}{p'}\geq\frac{n+\alpha-1}{r} \int_M\left(1+\frac{n-1}{n+\alpha-1}{\bf D}_\kappa(d_{x_0})\right)d_{x_0}^{\alpha-1}|u|^r\diff v_g.
	\end{equation*}
	Moreover the constant $\frac{n+\alpha-1}{r}$ is sharp.
\end{theorem}	
\begin{proof}
	The proof is similar to Theorem \ref{Uncertainty-theorem}, the difference is that we choose $H(s)=\frac{|s|^r}{r}$ for every $s\in\mathbb R$. The sharpness of $\frac{n+\alpha-1}{r}$ is based on the same scaling argument as before, the test function (in $\mathbb R^n$) in the last step being the Talenti-type function  $$u(x)=\left(1+|x|^{1+\frac{\alpha}{p-1}}\right)^\frac{p-1}{p-r},\quad\forall x\in \mathbb R^n.$$
	The rest is similar as before. 
\end{proof}

As we already pointed out in the introduction, various forms of $H$ and $G$ in Theorem \ref{theorem-main-0} produce well known  or new functional inequalities. In this spirit, we conclude the paper  by an unusual Caffarelli-Kohn-Nirenberg-type inequality, by inviting  the interested reader to build further functional inequalities through Theorems \ref{theorem-main-0} \& \ref{theorem-weighted-main}.

\begin{theorem}\label{thm-final}
Let $(M,g)$ be an $n$-dimensional Cartan-Hadamard manifold $(n\geq 2),$ such that ${\bf K}\leq \kappa<0$. Then for every  $u\in C_0^\infty(M)\setminus \{0\}$ and $c\in \mathbb R$, we have
\begin{equation}\label{ckn-fura}
	\int_M	|\nabla_gu|^{2} \diff v_g \geq \frac{\left(\displaystyle \int_M {\bf s}_c^2(u)\diff v_g\right)^2}{\displaystyle\int_M {\bf s}_c^2(2u)\diff v_g}(n-1)^2|\kappa|.
\end{equation}
In particular,  we also have  McKean's spectral gap estimate
$$\lambda_{1,2}(M)\geq \frac{(n-1)^2}{4}|\kappa|.$$
\end{theorem}
\begin{proof}
	We apply Theorem \ref{theorem-main-0}/(ii) with the choices $p=2$, $H(s)={\bf s}^2_c(s)$ for $s\in \mathbb R$, $G\equiv 1$, and $\rho=d_{x_0}$ for some $x_0\in M$. Since $\Delta_g d_{x_0}\geq (n-1){\bf ct}_\kappa(d_{x_0})\geq (n-1)\sqrt{-\kappa}$, see Theorem \ref{comparison-theorem}/(I)/(i), inequality \eqref{ckn-fura} follows at once. For $c=0$, the $L^2$-McKean spectral estimate follows.
\end{proof}

\noindent \textbf{Acknowledgment.} The authors thank the anonymous Referees for their valuable comments
and suggestions that significantly improved the presentation of the manuscript.\\

\noindent \textbf{Declaration}

\textit{Conflict of interest.} The authors state that there is no conflict of interest.

\end{document}